\documentclass[smallextended]{svjour3}
\smartqed

\usepackage{float,amssymb,amsmath}
\usepackage{graphicx}
\usepackage{mathptmx}

\floatstyle{ruled}
\newfloat{algorithm}{tbp}{loa}
\providecommand{\algorithmname}{Algorithm}
\floatname{algorithm}{\protect\algorithmname}
\usepackage{algorithmic}

\newcommand{\smu}{\scriptscriptstyle (\mu)}
\newcommand{\klein}[1]{\scriptscriptstyle #1}

\journalname{Numerical Algorithms}

\begin{document}

\title{Approximating the extreme Ritz values and upper bounds for the $A$-norm of the error in CG\thanks{This work was supported by the project 17-04150J of the Grant Agency of the Czech Republic.
}}
\author{G\'erard Meurant \and Petr Tich\'y}

\institute{G\'erard Meurant \at
              30 rue du sergent Bauchat, 75012 Paris, France \\
              \email{gerard.meurant@gmail.com}           
           \and
		Petr Tich\'y \at
              Faculty of Mathematics and Physics, Charles University\\
              Sokolovsk\'a 83, 186 75 Prague 8, Czech Republic \\
              \email{petr.tichy@mff.cuni.cz}           
}

\date{Received: date / Accepted: date}

\maketitle 
%
%
%

%

\begin{abstract}
In practical conjugate gradient (CG) computations it is important to monitor the quality of the approximate solution to $Ax=b$ so that the CG algorithm can be stopped when the required accuracy is reached.  The relevant convergence characteristics, like the $A$-norm of the error or the normwise backward error, cannot be easily computed. However, they can be estimated. Such estimates often depend on approximations of the smallest or largest eigenvalue of~$A$.

In the paper we introduce a new upper bound for the $A$-norm of the error, which is closely related to the Gauss-Radau upper bound, and discuss the problem of choosing the parameter $\mu$ which should represent a lower bound for the smallest eigenvalue of $A$.
The new bound has several practical advantages, the most important one is that it can be used as an approximation to the $A$-norm of the error even if $\mu$ is not exactly a lower bound for the smallest eigenvalue of $A$. In this case, $\mu$ can be chosen, e.g., as the smallest Ritz value or its approximation.
We also describe a very cheap algorithm, based on the incremental norm estimation technique, which allows to estimate the smallest and largest Ritz values during the CG computations. An improvement of the accuracy of these estimates of extreme Ritz values is possible, at the cost of storing the CG coefficients and solving a linear system with a tridiagonal matrix at each CG iteration.
Finally, we discuss how to cheaply approximate the normwise backward error.
The numerical experiments demonstrate the efficiency of the estimates of the extreme Ritz values, and show their practical use in error estimation in CG.



\end{abstract}
%
%

\section{Introduction}
The (preconditioned) Conjugate Gradient ((P)CG) algorithm by Hestenes and Stiefel \cite{HeSt1952} is now  considered as  the iterative method of
choice for solving linear systems $Ax=b$ with a real symmetric positive definite matrix~$A$. An important question to solve practical problems is to
know when to stop the iterations. Since, in CG, the norm of the residual vector $r_k=b-Ax_k$ (where $x_k$ is the approximate solution at iteration
$k$) is available, many CG codes use $\Vert r_k\Vert/\Vert r_0\Vert\le tol$ where $tol$ is a user-given threshold as a stopping criterion. This can be misleading
depending on the choice of the initial iterate $x_0$. A better stopping criterion is $\Vert r_k\Vert/\Vert b\Vert\le tol$. However, both criteria can
be misleading as it was already mentioned in \cite{HeSt1952}. Moreover, in many cases, the residual norm is oscillating making the use of these criteria
more problematic.

A more reliable stopping criterion could be based on 
the $A$-norm of the error 
$$
\Vert x - x_k \Vert_A =((x-x_k)^T A (x-x_k))^{1/2}.
$$
Mathematically, CG minimizes this quantity at each iteration $k$;
see \cite{HeSt1952}. In some linear systems arising from engineering problems the $A$-norm of the error corresponds to the energy norm and thus has a physical meaning.
Of course, in real-world problems the error and its norm are unknown. Therefore, this has lead to some research works for finding approximations or even lower
and upper bounds for the $A$-norm of the error. It turns out that the CG $A$-norm of the error is linked to a  Riemann-Stieltjes integral
for a discrete measure involving the distribution of the eigenvalues of $A$. Inspired by this connection already mentioned by Hestenes and Stiefel
\cite[p.~428]{HeSt1952},
 research on this topic was started by Gene Golub in the 1970's and continued throughout the years with several collaborators (e.g.,
G.~Dahlquist, S.~Eisenstat, S.~Nash, B.~Fischer, G.~Meurant, Z.~Strako\v{s}). The main idea is to approximate the Riemann-Stieltjes integral by Gauss or Gauss-Radau quadrature rules.
Since, in this case, the sign of the remainders of the quadrature rules are known, in theory this gives lower and upper bounds for the
$A$-norm of the error. These bounds can be used to design more
reliable stopping criteria than just using the relative norm of the residual. For details on these techniques, see \cite{DaEiGo1972,DaGoNa1979,FiGo1994,GoMe1994,GoSt1994,GoMe1997}. This research was summarized in \cite{B:Me2006} and \cite{B:GoMe2010}.
More recently,  some simpler and improved formulas for the computation of the bounds on the $A$-norm of the error were provided in \cite{MeTi2013}.

The techniques used in \cite{GoMe1994,GoSt1994,StTi2002,StTi2005} to compute lower or upper bounds use a positive integer $d$ which is called the delay, in such a way that, at CG iteration $k+d$, an estimate of the $A$-norm of the error at iteration $k$ is obtained. The larger the delay is, the better 
are the bounds at iteration $k$.
However, even when using these techniques, the situation is still not completely satisfactory. Obtaining an upper bound with the Gauss-Radau quadrature rule needs to have a prescribed parameter which should represent a lower bound for
the smallest eigenvalue of the (preconditioned) system matrix. This may not be readily available to the user. 
Moreover, some numerical examples have shown that, even if we have a good lower bound for the
smallest eigenvalue, the quality of the Gauss-Radau upper bound may deteriorate when the $A$-norm of the error becomes small.
Sometimes, it is also useful to compute an approximation of the matrix 2-norm
if the user wants to compute an estimate of the normwise backward error, see \cite{OePr1964,RiGa1967}, or to approximate the ultimate level of accuracy, or the condition number of the (preconditioned) system matrix.

The goal of this paper is to discuss and address these issues to obtain 
cheap approximations to the smallest and largest eigenvalues of the (preconditioned) system matrix during the CG computations, and to use them 
in estimating convergence characteristics like the $A$-norm of the error or the normwise backward error.
In particular, we introduce a new upper bound for the $A$-norm of the error which is less sensitive to the choice of the approximation to the smallest eigenvalue, and suggest an approximate upper bound which does not require any a priori information about the smallest eigenvalue.

The paper is organized as follows. In Section~\ref{sec:CG} we recall the Lanczos and CG algorithms as well as some relations which show the links between
CG and Gauss quadrature. Section~\ref{sec:bounds} is concerned with the Gauss-Radau upper bound and the derivation of a new upper bound. In Section~\ref{sec:bcsstk01} we present a numerical example that shows the troubles that may happen with the Gauss-Radau upper bounds, and a possible potential 
of the new upper bound which is not sensitive to the choice of the approximation to the smallest eigenvalue.
In Section~\ref{sec:Ritz_val} we address  the problem of computing estimates of the smallest and largest eigenvalues of $A$. This is done by using incremental estimates of norms of bidiagonal matrices. These algorithms can be useful in a more general setting than computing bounds for the CG error norms. 
In Sections~\ref{sec:radau} and \ref{sec:be}
these results are used to approximate the Gauss-Radau upper bound and the normwise backward error. Section~\ref{sec:exp} illustrates numerically the quality of approximations to the smallest and largest eigenvalues, and their use in approximating the normwise backward error and the $A$-norm of the error. Finally, in Section~\ref{sec:concl} we give some conclusions and perspectives.

\section{The Lanczos and CG algorithms}
 \label{sec:CG}

Given a starting vector $v\in\mathbb{R}^{N}$ and a symmetric matrix $A\in\mathbb{R}^{N\times N}$,
one can consider a sequence of nested subspaces
\[
\mathcal{K}_{k}(A,v)\equiv\mathrm{span}\{v,Av,\dots,A^{k-1}v\},\qquad k=1,2,\dots,
\]
called Krylov subspaces. The dimension of these subspaces is increasing
up to an index $n\leq N$ called the \emph{grade of $v$ with respect to
$A$}, at which the maximal dimension is attained, and $\mathcal{K}_{n}(A,v)$
is invariant under multiplication with $A$.
\begin{algorithm}[ht]
\caption{Lanczos algorithm}
\label{alg:lanczos}

\begin{algorithmic}[0]

\STATE \textbf{input} $A$, $v$

\STATE $\widetilde{\beta}_{0}=0$, $v_{0}=0$

\STATE $v_{1}=v/\|v\|$

\FOR{$k=1,\dots$}

\STATE $w=Av_{k}-\widetilde{\beta}_{k-1}v_{k-1}$

\STATE $\widetilde{\alpha}_{k}=v_{k}^{T}w$

\STATE $w=w-\widetilde{\alpha}_{k}v_{k}$

\STATE $\widetilde{\beta}_{k}=\|w\|$

\STATE $v_{k+1}=w/\widetilde{\beta}_{k}$

\ENDFOR

\end{algorithmic}
\end{algorithm}
Assuming that $k<n$, the Lanczos algorithm (Algorithm~\ref{alg:lanczos})
computes an orthonormal basis $v_{1},\dots,v_{k+1}$ of the Krylov
subspace $\mathcal{K}_{k+1}(A,v)$. The basis vectors $v_{j}$ of unit norm satisfy
the matrix relation
\[
AV_{k}=V_{k}T_{k}+\widetilde{\beta}_{k}v_{k+1}e_{k}^{T}
\]
where $V_{k}=[v_{1}\cdots v_{k}]$, $e_k$ denotes the $k$th column of the identity matrix, and

\[
T_{k}=\left[\begin{array}{cccc}
\widetilde{\alpha}_{1} & \widetilde{\beta}_{1}\\
\widetilde{\beta}_{1} & \ddots & \ddots\\
 & \ddots & \ddots & \widetilde{\beta}_{k-1}\\
 &  & \widetilde{\beta}_{k-1} & \widetilde{\alpha}_{k}
\end{array}\right]
\]
is the $k\times k$
symmetric tridiagonal matrix of the recurrence coefficients computed
in Algorithm~\ref{alg:lanczos}.
The coefficients $\widetilde{\beta}_{j}$ being positive, $T_{k}$
is a Jacobi matrix. The Lanczos algorithm works for any symmetric
matrix, but if $A$ is positive definite, then $T_{k}$ is positive
definite as well.

When solving a system of linear equations $Ax=b$ with a real symmetric positive definite matrix $A$, the CG method (Algorithm \ref{alg:cg})
\begin{algorithm}[ht]
\caption{Conjugate Gradients} \label{alg:cg}

\begin{algorithmic}[0]

\STATE \textbf{input} $A$, $b$, $x_{0}$

\STATE $r_{0}=b-Ax_{0}$

\STATE $p_{0}=r_{0}$

\FOR{$k=1,\dots$ until convergence}

\STATE $\gamma_{k-1}=\frac{r_{k-1}^{T}r_{k-1}}{p_{k-1}^{T}Ap_{k-1}}$

\STATE $x_{k}=x_{k-1}+\gamma_{k-1}p_{k-1}$

\STATE $r_{k}=r_{k-1}-\gamma_{k-1}Ap_{k-1}$

\STATE $\delta_{k}=\frac{r_{k}^{T}r_{k}}{r_{k-1}^{T}r_{k-1}}$

\STATE $p_{k}=r_{k}+\delta_{k}p_{k-1}$

\ENDFOR

\end{algorithmic}
\end{algorithm}
can be used. Mathematically, the CG iterates $x_{k}$ minimize
the $A$-norm of the error over the manifold $x_{0}+\mathcal{K}_{k}(A,r_{0})$,
\[
\|x-x_{k}\|_{A}=\min_{y\in x_{0}+\mathcal{K}_{k}(A,r_{0})}\|x-y\|_{A},
\]
and the residual vectors $r_{k}=b-Ax_{k}$ are proportional to the
Lanczos vectors $v_{j}$,
\[
v_{j+1}=(-1)^{j}\frac{r_{j}}{\|r_{j}\|}\,,\qquad j=0,\dots,k.
\]
Thanks to this close relationship between the CG and Lanczos algorithms
it can be shown (see, for instance \cite{B:Me2006}) that the recurrence
coefficients computed in both algorithms are connected via
\begin{equation}
\widetilde{\beta}_{k}=\frac{\sqrt{\delta_{k}}}{\gamma_{k-1}},\quad\widetilde{\alpha}_{k}=\frac{1}{\gamma_{k-1}}+\frac{\delta_{k-1}}{\gamma_{k-2}},\quad\delta_{0}=0,\quad\gamma_{-1}=1.\label{eq:CGLanczos}
\end{equation}
 Writing these formulas in matrix form, we get
$$
T_{k}=L_{k}L_{k}^{T},\qquad L_{k}^{T}=\left[\begin{array}{cccc}
\frac{1}{\sqrt{\gamma_{0}}} & \sqrt{\frac{\delta_{1}}{\gamma_{0}}}\\
 & \ddots & \ddots\\
 &  & \ddots & \sqrt{\frac{\delta_{k-1}}{\gamma_{k-2}}}\\
 &  &  & \frac{1}{\sqrt{\gamma_{k-1}}}
\end{array}\right].
$$
In other words, CG computes implicitly the Cholesky factorization
of the Jacobi matrix $T_{k}$ generated by the Lanczos algorithm.
Hence, the eigenvalues of $T_{k}$ (the so-called Ritz values) are equal to the squared singular
values of the upper bidiagonal matrix $L_{k}^{T}$.

It is well known that the reduction of the squared $A$-norm of the error from iteration $k-1$ to iteration $k$ is given by  $\gamma_{k-1}\|r_{k-1}\|^{2}$;
see \cite[relation~(6:1)]{HeSt1952}. As a consequence
\begin{equation}
\|x-x_{0}\|_{A}^{2}=\sum_{j=0}^{k-1}\gamma_{j}\|r_{j}\|^{2}+\|x-x_{k}\|_{A}^{2}.\label{eq:GQ}
\end{equation}
The relation~\eqref{eq:GQ} represents the basis for the quadrature-based estimation of the $A$-norm of the error in the CG method
\cite{GoMe1994,GoSt1994,GoMe1997,StTi2002,StTi2005,MeTi2013,MeTi2014}. In more details, let $A=Q\Lambda Q^T$ be the spectral decomposition of $A$, with
$Q=[q_1,\dots,q_N]$ orthonormal and $\Lambda=\mathrm{diag}(\lambda_1,\dots,\lambda_N)$, the $\lambda_i$'s, $i=1,\dots,N$ being the eigenvalues of $A$. For simplicity of notation we assume that the eigenvalues of $A$
are distinct and ordered as $\lambda_1 < \lambda_2 < \dots < \lambda_N$. Let us define the weights $\omega_{i}$ by
\begin{equation}
 \omega_{i}\equiv\frac{(r_{0},q_{i})^{2}}{\| r_0 \|^2}\qquad\mbox{so that}\qquad\sum_{i=1}^{N}\omega_{i}=1\,,\label{17}
\end{equation}
and the (nondecreasing) stepwise constant distribution function $\omega(\lambda)$ with a finite number of points of increase
$\lambda_{1},\lambda_{2},\dots,\lambda_{N}$,
\begin{equation}
 \omega(\lambda)\equiv\;\left\{ \;
 \begin{array}{rcl}
 0 & \textnormal{for} & \lambda<\lambda_{1}\,,\\[1mm]
 \sum_{j=1}^{i}\omega_{j} & \textnormal{for} & \lambda_{i}\leq\lambda<\lambda_{i+1}\,,\quad1\leq i\leq N-1\,,\\[1mm]
 1 & \textnormal{for} & \lambda_{N}\leq\lambda\,.
 \end{array}\right.\,\label{eq:schema_hermit}
\end{equation}
 Having the distribution function $\omega(\lambda)$ and an interval
$\langle\zeta,\xi\rangle$ such that $\zeta<\lambda_{1}<\lambda_{2}<\dots<\lambda_{N}<\xi$, for any continuous function $f$, one can define  the
Riemann-Stieltjes integral (see, for instance \cite{B:GoMe2010})
\begin{equation}
 \int_{\zeta}^{\xi}f(\lambda)\, d\omega(\lambda) = \sum_{i=1}^{N}\omega_{i}f(\lambda_{i}).\label{19}
\end{equation}
For the integrated function defined as $f(\lambda)=\lambda^{-1}$, we obtain the integral representation of the squared initial $A$-norm of the error
\begin{eqnarray*}
\| x-x_0\|_A^2 &=& r_0^TA^{-1} r_0
= (Q^Tr_0)^T\Lambda^{-1} (Q^T r_0) \\ &=&
\|r_0\|^2\sum_{j=1}^n \lambda_j^{-1} \omega_j=
\|r_0\|^2 \int_{\zeta}^{\xi}\lambda^{-1}\, d\omega(\lambda).
\end{eqnarray*}
Finally, using the optimality of CG, it can be shown that the formula~\eqref{eq:GQ} represents the scaled $k$-point Gauss quadrature rule for
approximating the Riemann-Stieltjes integral of the function $f(\lambda)=\lambda^{-1}$, with the scaled positive reminder $\|x-x_k\|^2_A$. The
scaling factor is $\|r_{0}\|^{-2}$. Various modified quadrature rules can be used to obtain other approximations to the integral, possibly also with a negative reminder. Such rules usually require some a priori information about the spectrum of $A$. For a summary, see, e.g., the book
\cite{B:GoMe2010}.

\section{Quadrature-based bounds and a new upper bound
 \label{sec:bounds}}

In this section we concentrate on two simple upper bounds. To summarize some results of
\cite{GoSt1994,GoMe1994,StTi2002}, and \cite{MeTi2013} related to
the Gauss and Gauss-Radau quadrature bounds for the $A$-norm of the
error in~CG, it has been shown that
\begin{equation}
 \label{eq:basic}
\gamma_{k}\|r_{k}\|^{2}<\|x-x_{k}\|_{A}^{2}<\gamma_{k}^{\smu}\|r_{k}\|^{2}
\end{equation}
 where
\begin{equation}
\gamma_{k+1}^{{\scriptscriptstyle (\mu)}}=\frac{\left(\gamma_{k}^{{\scriptscriptstyle (\mu)}}-\gamma_{k}\right)}{\mu\left(\gamma_{k}^{{\scriptscriptstyle (\mu)}}-\gamma_{k}\right)+\delta_{k+1}},\qquad\gamma_{0}^{{\scriptscriptstyle (\mu)}}=\frac{1}{\mu},\label{eq:gamma}
\end{equation}
$k<n-1$, and $\mu$ such that $0<\mu\leq\lambda_{\min}$.
Note that in the special case $k=n-1$  since
$\|x-x_{n}\|_{A}^{2}=0$, we get $\|x-x_{n-1}\|_{A}^{2}=\gamma_{n-1}\|r_{n-1}\|^{2}$. If the initial residual $r_0$ has a nontrivial component in the eigenvector corresponding to $\lambda_{\min}$, then $\lambda_{\min}$ is also an eigenvalue of $T_n$. If in addition $\mu$ is chosen such that $\mu=\lambda_{\min}$, then $\gamma_{n-1}=\gamma_{n-1}^{\smu}$ and the second strict inequality in \eqref{eq:basic} changes to equality.

The simple updating formula \eqref{eq:gamma} was first presented in~\cite{MeTi2013}. Following the idea of \cite{GoSt1994} and~\cite{StTi2002}, we
can improve the lower and upper bounds in \eqref{eq:basic} by considering quadrature rules \eqref{eq:GQ} at iterations $k$ and $k+d$ for some integer $d>0$
which is called the delay. Then, we get the formula
\begin{equation}
 \label{eq:HS}
\|x-x_{k}\|_{A}^{2}=\sum_{j=k}^{k+d-1}\gamma_{j}\|r_{j}\|^{2}+\|x-x_{k+d}\|_{A}^{2},
\end{equation}
and one can bound the error norm at the iteration $k+d$ using \eqref{eq:basic} to obtain 
\begin{equation}
 \label{eq:delaylower}
 \sum_{j=k}^{k+d-1}\gamma_{j}\|r_{j}\|^{2}+\gamma_{k+d}\Vert r_{k+d}\Vert^2<
 \|x-x_{k}\|_{A}^{2} 
\end{equation}
and
\begin{equation}
 \label{eq:delayupper}
 \|x-x_{k}\|_{A}^{2} < \sum_{j=k}^{k+d-1}\gamma_{j}\|r_{j}\|^{2}+\gamma_{k+d}^{\smu}\Vert r_{k+d}\Vert^2.
\end{equation}
Note that \eqref{eq:delaylower} and \eqref{eq:delayupper} give a lower bound and an upper bound for the $A$-norm of the error at iteration $k$ when CG is already at iteration $k+d$ whence \eqref{eq:basic} provides
lower and upper bounds when CG is at iteration $k$.
In \cite{StTi2002} it has been shown that the identity \eqref{eq:HS} holds
(up to some small inaccuracies) also for numerically computed quantities in finite precision arithmetic, until the $A$-norm of the error reaches its ultimate level of accuracy. So, it can be used safely for estimating the $A$-norm of the actual error.

Mathematically, we will derive another upper bound for the
squared $A$-norm of the error, which is closely related to the Gauss-Radau upper bound.
This bound depends on the ratio
\[
\phi_{k}\equiv\frac{\left\Vert r_{k}\right\Vert ^{2}}{\left\Vert p_{k}\right\Vert ^{2}},
\]
which can be updated using a simple recurrence relation. In particular,
using $p_{k}=r_{k}+\delta_{k}p_{k-1}$ and the orthogonality between
$r_{k}$ and $p_{k-1}$ (local orthogonality), we obtain
\[
\left\Vert p_{k}\right\Vert ^{2}=\left\Vert r_{k}\right\Vert ^{2}\left(1+\delta_{k}\frac{\left\Vert p_{k-1}\right\Vert ^{2}}{\|r_{k-1}\|^{2}}\right),
\]
and, therefore,
\begin{equation}
\phi_{k}=\frac{\phi_{k-1}}{\phi_{k-1}+\delta_{k}},\qquad\phi_{0}=1.\label{eq:phi}
\end{equation}
Hence $\phi_k$ can be updated cheaply without computing the norm of $p_k$ which is not readily available in CG.
From \eqref{eq:phi} and by induction, it follows that
$$
	\phi_k^{-1} = 1 + \frac{\Vert r_k \Vert^2}{\Vert r_{k-1} \Vert^2} \phi_{k-1}^{-1} = \Vert r_k \Vert^2 
	\sum_{j=0}^{k}\|r_{j}\|^{-2}
$$
and hence
\begin{equation}
\left\Vert r_{k}\right\Vert ^{2}\phi_{k}=\left(\sum_{j=0}^{k}\|r_{j}\|^{-2}\right)^{-1};\label{eq:minres}
\end{equation}
see also \cite[Theorem~5:3]{HeSt1952}. Note that mathematically,
the quantity \eqref{eq:minres} can be interpreted as the norm of
the residual vector determined by the minimal residual method; see,
e.g., \cite[Theorem~3.5]{EiEr2001}. In finite precision arithmetic, 
the quantity \eqref{eq:minres} cannot be, in general, interpreted
as the norm of the residual vector generated by some minimal residual
method. We remark that the quantity $\phi_{k}$ appears also as a coefficient
in strategies for residual smoothing \cite{GuRo2001,GuRo2001a}. In
particular, one can compute the smoothed residual $r_{k}^{S}$ and
the corresponding approximation $x_{k}^{S}$ using the recurrences
\begin{eqnarray*}
r_{k}^{S} & = & \left(1-\phi_{k}\right)r_{k-1}^{S}+\phi_{k}r_{k},\quad x_{k}^{S}=\left(1-\phi_{k}\right)x_{k-1}^{S}+\phi_{k}x_{k}.
\end{eqnarray*}

\noindent The new upper bound is as follows.

\begin{theorem}
Let $0<\mu\leq\lambda_{\min}$ be given. The approximations $x_{k}$, $k<n$,
generated by the CG method satisfy
\begin{equation}
\|x-x_{k}\|_{A}^{2} \,<\, \frac{\|r_{k}\|^{2}}{\mu}\frac{\|r_{k}\|^{2}}{\|p_{k}\|^{2}},\label{eq:newbound}
\end{equation}
and the bound is decreasing with increasing $k$.
\end{theorem}
\begin{proof}
Based on \eqref{eq:basic} it is sufficient to show that $\mu\gamma_{k}^{{\scriptscriptstyle (\mu)}}\leq\phi_{k}$.
We will prove it by induction. The inequality holds for $k=0$. Using
the induction hypothesis, \eqref{eq:gamma}, and \eqref{eq:phi}
we obtain, for $k<n-1$,
\[
\mu\gamma_{k+1}^{{\scriptscriptstyle (\mu)}}=\frac{\mu\left(\gamma_{k}^{{\scriptscriptstyle (\mu)}}-\gamma_{k}\right)}{\mu\left(\gamma_{k}^{{\scriptscriptstyle (\mu)}}-\gamma_{k}\right)+\delta_{k+1}}<\frac{\mu\gamma_{k}^{{\scriptscriptstyle (\mu)}}}{\mu\gamma_{k}^{{\scriptscriptstyle (\mu)}}+\delta_{k+1}}\leq\frac{\phi_{k}}{\phi_{k}+\delta_{k+1}}=\phi_{k+1}.
\]
Recall that $\gamma_{k}^{{\scriptscriptstyle (\mu)}}-\gamma_{k}$ is
positive because of \eqref{eq:basic}. Finally, using \eqref{eq:minres},
the bound \eqref{eq:newbound} is monotonically decreasing
with increasing $k$.
\end{proof}

The tightness of the bound \eqref{eq:newbound} can further be improved 
when using a delay~$d$, similarly as in \eqref{eq:delayupper}. First, the proof of the previous theorem also shows that the Gauss-Radau upper bound presented in \eqref{eq:basic} can be bounded from above by
\begin{equation}\label{eq:newbound2}
	\gamma_{k}^{{\scriptscriptstyle (\mu)}} \| r_k \|^2 < 
	\frac{\|r_{k}\|^{2}}{\mu}\frac{\|r_{k}\|^{2}}{\|p_{k}\|^{2}}.
\end{equation}
Second, combining \eqref{eq:delayupper} and \eqref{eq:newbound2} we can get 
an improved upper bound 
\begin{equation}\label{eq:newbound3}
	\|x-x_{k}\|_{A}^{2} <  \sum_{j=k}^{k+d-1}\gamma_{j}\|r_{j}\|^{2}+\frac{\|r_{k+d}\|^{2}}{\mu}\frac{\|r_{k+d}\|^{2}}{\|p_{k+d}\|^{2}}.
\end{equation}

In practical computations, the parameter $\mu$ has to be determined.
This represents a nontrivial task.

\section{A numerical example: The choice of $\mu$}
 \label{sec:bcsstk01}

As an example that can demonstrate the difficulties to compute accurate upper bounds for the $A$-norm of the error, we consider the 
matrix {\tt bcsstk01} from the set BCSSTRUC1 in the Harwell-Boeing collection, which can be obtained from 
the Matrix Market\footnote{{\tt
http://math.nist.gov/MatrixMarket}} 
or from the SuiteSparse Matrix Collection\footnote{
{\tt
https://sparse.tamu.edu/}}. It is a small stiffness matrix of order $48$
 arising from dynamic analysis in structural engineering with $400$ nonzero entries. Its condition number is $\kappa(A)=8.8234\times 10^5$.
The smallest eigenvalue $\lambda_{\min}(A)=3.417267562666500\times 10^3$ was computed in extended precision and rounded to double precision.
The right-hand side $b$ has been chosen such that $b$ has equal
components in the eigenvector basis, and such that $\|b\|=1$.

The linear system $Ax=b$ is difficult to solve with CG without a preconditioner. We have to perform around $180$ iterations to reach the maximum attainable accuracy when the
matrix is only of order $48$. There is a long phase of quasi-stagnation of the $A$-norm of the error that last almost $100$ iterations as one can see in
\figurename~\ref{fig-1}. Denote
\begin{equation}\label{eq:denote}
 u_k^{\smu} \equiv \sqrt{\gamma_k^{\smu}} \Vert r_k\Vert, \qquad \pmb{u}_k^{\smu} \equiv \sqrt{\frac{\phi_k}{\mu}}\, \| r_k \|
\end{equation}
the bounds which correspond to \eqref{eq:basic} and \eqref{eq:newbound} (without any delay $d=0$).

\figurename~\ref{fig-1} displays the $A$-norm of the error (dotted curve), the bounds
$u_k^{\smu}$ for
different values of $\mu$ equal to $\lambda_{\min}/(1+10^{-m}), m=2,...,14$ (dashed curves), and the new upper bounds $\pmb{u}_k^{\smu}$ (thick solid curves).
The closer $\mu$ is to $\lambda_{\min}$ the better is the upper bound $u_k^{\smu}$ of the $A$-norm of the error. However, below a level of
approximately $10^{-8}$ all the values of $\mu$ in our experiment give visually the same upper bound $u_k^{\smu}$ which is not very close to the
$A$-norm of the error. We can also observe
 that the new upper bound $\pmb{u}_k^{\smu}$ is insensitive to the choice of $\mu$ and gives an envelope of the Gauss-Radau upper bounds~$u_k^{\smu}$.

\begin{figure}[!htbp]
\centering
 \includegraphics[width=8cm]{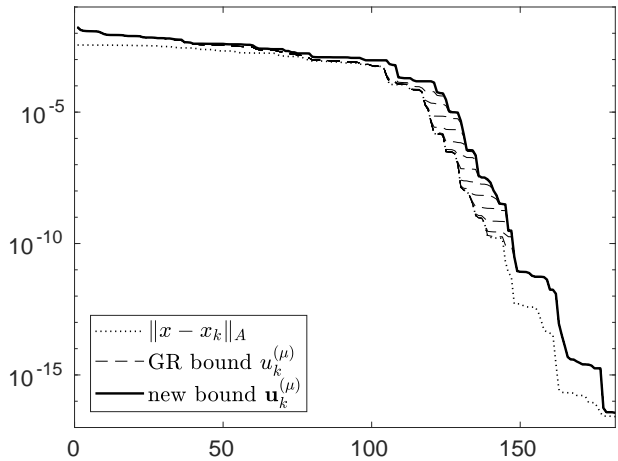}
\caption{{\tt bcsstk01}, $u_k^{\smu}$ and $\pmb{u}_k^{\smu}$, $\mu=\lambda_{\min}/(1+10^{-m}), m=2,...,14$} \label{fig-1}
\end{figure}

\figurename~\ref{fig-2} shows the ``upper bounds'' $u_k^{\smu}$ for values of $\mu$ which are larger than but close to $\lambda_{\min}$; $\mu =
\lambda_{\min}/(1-10^{-m}), m=2,4,6,...,14$. We use quotes since, as one can see, we do not obtain an upper bound in general, even though we are
close to $\lambda_{\min}$. If $\mu$ is chosen to be larger than $\lambda_{\min}$, then, at some point, the coefficient $\gamma_k^{\smu}$ can even be
negative. In such cases we use $|\gamma_k^{\smu}|$, and emphasize the corresponding value by a dot.
\begin{figure}[!htbp]
\centering
 \includegraphics[width=8cm]{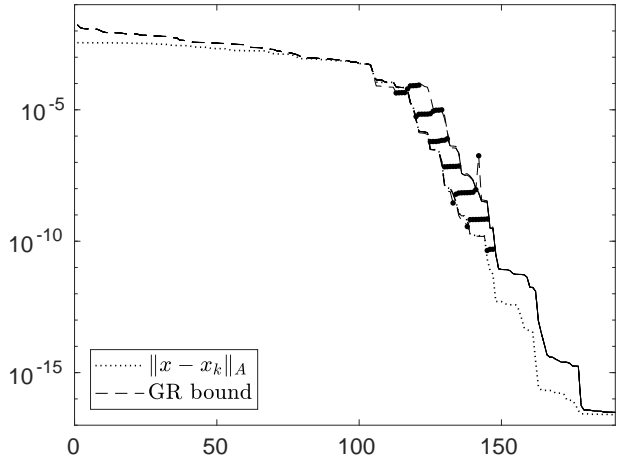}
\caption{{\tt bcsstk01}, $\sqrt{\left|\gamma_k^{\smu}\right|}\, \Vert r_k\Vert$, $\mu = \lambda_{\min}/(1-10^{-m}), m=2,4,6,,...,14$ } \label{fig-2}
\end{figure}
In \figurename~\ref{fig-2} we do not plot the new bound $\pmb{u}_k^{\smu}$.
However, 
from its definition and the assumption that $\mu \approx\lambda_{\min}$
it follows that $\pmb{u}_k^{\smu}$  will stay visually the same as in 
\figurename~\ref{fig-1}.

In summary, the node $\mu$ should satisfy
$\mu\leq\lambda_{\min}$, and, simultaneously, it should closely approximate $\lambda_{\min}$, otherwise the Gauss-Radau upper bound
$u_k^{\smu}$ would be a poor approximation of the $A$-norm of the error.
If the smallest eigenvalue is known in advance, then the bound $u_k^{\smu}$ can give very good results until some level of accuracy of the error norm
(in our case $10^{-8}$) is reached. Below this level, the bounds $u_k^{\smu}$ and $\pmb{u}_k^{\smu}$ visually coincide, and are far away from the $A$-norm of the norm.

If the parameter $\mu$ has to be determined, possibly in some adaptive way, then we can expect troubles. First, one cannot hope in general to get a
very accurate approximation of the smallest eigenvalue without too much work. Second, there is usually no guarantee that the condition $\mu\leq\lambda_{\min}$ is
satisfied. Typically, the best we can get from the Lanczos process are  the Ritz values (eigenvalues of $T_k$) which can approximate the eigenvalues
of~$A$. However, Ritz values provide only upper bounds on $\lambda_{\min}$, and some heuristics (e.g., multiplication by a safety constant) have to
be used to obtain $\mu$ with the desired properties. As we have seen in the numerical example, the value $u_k^{\smu}$ can be very sensitive to small
perturbations of~$\mu$. Then, using a heuristic can strongly influence the approximation properties of $u_k^{\smu}$, and cause numerical troubles in
computation of $u_k^{\smu}$ if $\mu>\lambda_{\min}$. On the other hand, the new bound $\pmb{u}_k^{\smu}$ can be computed without any troubles also
for $\mu>\lambda_{\min}$. If in addition $\mu\approx\lambda_{\min}$, then either $\pmb{u}_k^{\smu}$ represents an upper bound, or, it is an 
approximation of the $A$-norm of the error. In other words, an approximation of the smallest Ritz value can be used as a heuristic for the bound
$\pmb{u}_k^{\smu}$.

\section{Approximating the extreme Ritz values}
 \label{sec:Ritz_val}

In this section we develop efficient algorithms for the incremental
approximation of the smallest and largest Ritz values. This information
can be used not only in the error approximation techniques based on
various modified quadrature rules (see, e.g., \cite{GoMe1994,GoMe1997,MeTi2013}),
but also to approximate the 2-norm of $A$ or the condition number
of $A$. Note that an approximation of $\|A\|$ is needed in
estimating the maximum attainable accuracy (see \cite{Gr1997}) or in the computation of
the normwise backward error (see \cite{RiGa1967}).

As already mentioned, Jacobi matrices $T_{k}$ and the lower bidiagonal matrices $L_{k}$
which appear in CG are related through $T_{k}=L_{k}L_{k}^{T}$. In
particular, it holds that
\begin{equation}
\lambda_{\max}(T_{k})=\left\Vert L_{k}\right\Vert ^{2},\qquad\lambda_{\min}(T_{k})=\left\Vert L_{k}^{-1}\right\Vert ^{-2}.\label{eq:why}
\end{equation}
Hence, one can approximate the extreme eigenvalues of $T_{k}$ using
 incremental norm estimation applied to the upper triangular matrices
$L_{k}^{T}$ and $L_{k}^{-T}$. Although we are mainly motivated by
the approximation of the extreme Ritz values in CG, we consider the
problem of incremental norm estimation of bidiagonal matrices and
their inverses by itself, since it can be useful also in other algorithms
involving bidiagonal matrices.

\subsection{The eigenvalues and eigenvectors of a $2\times2$ symmetric matrix}

An important ingredient of incremental norm estimation is the fact
that the eigenvalues and eigenvectors of a $2\times2$ symmetric matrix
are known explicitly. Consider a matrix of the form
\begin{equation}
\left[\begin{array}{cc}
\rho & \sigma\\
\sigma & \tau
\end{array}\right].\label{eq:B}
\end{equation}
The two eigenvalues of \eqref{eq:B} are given by
\[
\lambda_{+}=\frac{1}{2}\left(\rho+\tau+\chi\right),\qquad\lambda_{-}=\frac{1}{2}\left(\rho+\tau-\chi\right)
\]
where
\begin{equation}
 \label{eq:chi}
\chi^{2}=(\rho-\tau)^{2}+4\sigma{}^{2}.
\end{equation}
If $\sigma\neq0$, the matrix of unnormalized eigenvectors is given
by
\[
\left[\begin{array}{cc}
\rho-\tau+\chi & \rho-\tau-\chi\\
2\sigma & 2\sigma
\end{array}\right].
\]
For more details see \cite[p.306]{DuVo2002}, \cite[p.166]{B:GoMe2010}.

\subsection{Incremental estimation of the norms of upper triangular matrices\label{sec:incremental}}

To approximate the maximum singular value of an upper triangular matrix,
we use an incremental estimator proposed in \cite{DuVo2002}. The
algorithm is based on incremental improvement of an approximation
of the right singular vector that corresponds to the maximum singular
value. In~\cite{DuTu2014} it has been shown that this technique
tends to be superior, with respect to approximating maximum singular
values, to the original incremental technique proposed in \cite{Bi1990}.
In the following we recall the basic idea of the incremental norm
estimation and reformulate slightly the algorithm so that it can
be efficiently applied to upper bidiagonal matrices and their inverses.

Let $R\in\mathbb{R}^{k\times k}$ be an upper triangular matrix and let $z$ be its approximate (or exact) maximum right singular vector. Let
\begin{equation}
\widehat{R}=\left[\begin{array}{cc}
R & v\\
 & \eta
\end{array}\right],\qquad v\in\mathbb{R}^{k},\quad\eta\in\mathbb{R},\label{eq:Rext}
\end{equation}
and consider the new approximate maximum right singular vector in
the form
\begin{equation}
\widehat{z}=\left[\begin{array}{c}
sz\\
c
\end{array}\right],\label{eq:z}
\end{equation}
where $s^{2}+c^{2}=1$. The parameters $s$ and $c$ are chosen such that the norm of the vector
$\widehat{R}\widehat{z}$ is maximal. It holds that
\[
\|\widehat{R}\widehat{z}\|^{2}=\left[\begin{array}{c}
s\\
c
\end{array}\right]^{T}\left[\begin{array}{cc}
\rho & \sigma\\
\sigma & \tau
\end{array}\right]\left[\begin{array}{c}
s\\
c
\end{array}\right]
\]
where
\[
\rho=\|Rz\|^{2},\qquad\sigma=v^{T}Rz,\qquad\tau=v^{T}v+\eta^{2}.
\]
Hence, to maximize $\|\widehat{R}\widehat{z}\|^{2}$, we need to determine
the maximum eigenvalue of the symmetric $2\times2$ matrix \eqref{eq:B},
and the corresponding eigenvector. Using the previous results
\begin{equation}
\left[\begin{array}{c}
s\\
c
\end{array}\right]=\frac{u}{\|u\|},\qquad u=\left[\begin{array}{c}
\rho-\tau+\chi\\
2\sigma
\end{array}\right],\label{eq:sc}
\end{equation}
and
\begin{equation*}
\lambda_{+}=\frac{\rho+\tau+\chi}{2},\qquad\chi^{2}=\left(\rho-\tau\right)^{2}+4\sigma{}^{2}.
\end{equation*}
Note that if $\sigma=0$, the formula for the eigenvector that corresponds
to $\lambda_{+}$ is still valid. Next, it holds that
\[
\|u\|^{2}=2(\chi^{2}+\left(\rho-\tau\right)\chi),
\]
and, therefore, from \eqref{eq:chi},
\[
c^{2}=\frac{2\sigma^{2}}{\chi^{2}+\left(\rho-\tau\right)\chi}=\frac{1}{2}\frac{\chi^{2}-\left(\rho-\tau\right)^{2}}{\chi^{2}+\left(\rho-\tau\right)\chi}=\frac{1}{2}\left(1-\frac{\rho-\tau}{\chi}\right).
\]
We can also express $\|\widehat{R}\widehat{z}\|^{2}$ in a more convenient form
\begin{eqnarray*}
\|\widehat{R}\widehat{z}\|^{2} & = & \frac{\rho+\tau+\chi}{2}=\rho+\frac{\chi}{2}\left(1-\frac{\rho-\tau}{\chi}\right)=\rho+\chi c^{2}.
\end{eqnarray*}
To compute $\widehat{z}$, we still need to determine the signs of $s$ and $c$. From \eqref{eq:sc} it follows that $s\geq0$ and $c$ has the same sign as
$\sigma.$ Therefore,
\[
s=\sqrt{1-c^{2}},\qquad c=|c|\mathrm{sign}(\sigma).
\]
Using the subscript $k$, we can formulate Algorithm~\ref{alg:incremental}
for the incremental norm estimation~of
\begin{equation}
R_{k+1}=\left[\begin{array}{cc}
R_{k} & v_{k}\\
 & \eta_{k}
\end{array}\right],\qquad v_{k}\in\mathbb{R}^{k},\quad\eta_{k}\in\mathbb{R},\label{eq:Rk}
\end{equation}
where $R_k$ in Algorithm~\ref{alg:incremental} is a principal submatrix of $R_{k+1}$.

\begin{algorithm}[ht]
\caption{Incremental estimation of $\|R_{k}\|^{2}$}
\label{alg:incremental}

\begin{algorithmic}[0]

\STATE \textbf{input} matrices $R_{k}$

\STATE $z_{1}=1$,

\FOR{$k=1,\dots$}

\STATE

\STATE \texttt{\% $\dots$ Compute the entries of the $2\times2$
matrix.}

\STATE

\STATE $\rho_{k}=\|R_{k}z_{k}\|^{2},\ \sigma_{k}=v_{k}^{T}R_{k}z_{k},\ \tau_{k}=v_{k}^{T}v_{k}+\eta_{k}^{2}$

\STATE

\STATE \texttt{\% $\dots$ Compute the new estimate $\rho_{k+1}$. }

\STATE

\STATE $\chi_{k}^{2}=\left(\rho_{k}-\tau_{k}\right)^{2}+4\sigma_{k}^{2},\ c_{k}^{2}=\frac{1}{2}\left(1-\frac{\rho_{k}-\tau_{k}}{\chi_{k}}\right),\ \rho_{k+1}=\rho_{k}+\chi_{k}c_{k}^{2}$

\STATE

\STATE \texttt{\% $\dots$ If required, compute $z_{k+1}$.}

\STATE

\STATE $s_{k}=\sqrt{1-c_{k}^{2}},\ c_{k}=|c_{k}|\,\mathrm{sign}(\sigma_{k}),\ z_{k+1}=\left[\begin{array}{c}
s_{k}z_{k}\\
c_{k}
\end{array}\right]$

\ENDFOR

\end{algorithmic}
\end{algorithm}

Note that if we start the algorithm with $z_{1}=1$, then $\rho_{1}=r_{1,1}^{2}$, and $\rho_{2}$ is equal to $\|R_{2}\|^{2}$. In more details, it
holds that
\begin{eqnarray*}
\rho_{2} & = & \rho_{1}+\chi_{1}c_{1}^{2}\;=\;\frac{1}{2}\left(\chi_{1}+\rho_{1}+\tau_{1}\right)\\
 & =\frac{1}{2} & \left(r_{1,1}^{2}+r_{2,2}^{2}+r_{1,2}^{2}+\sqrt{\left(r_{1,1}^{2}-r_{2,2}^{2}-r_{1,2}^{2}\right)^{2}+4r_{1,1}^{2}r_{1,2}^{2}}\right)\;=\;\|R_{2}^{T}R_{2}\|.
\end{eqnarray*}
As we will see in the following, if $R_{k}$ is upper bidiagonal, it is possible to incrementally estimate $\left\Vert R_{k}\right\Vert $ and
$\left\Vert R_{k}^{-1}\right\Vert $ in a very efficient way, without storing the coefficients of the matrix~$R_{k}$ and even without storing the
approximate right singular vectors $z_{k}$. In particular, we will be able to find simple updating formulas for $\sigma_{k}$ and $\tau_{k}$ which are
then used in the updating formula for $\rho_{k+1}$.

\subsection{Specialization to upper bidiagonal matrices}
\label{sec:NormB}
Consider a bidiagonal matrix $B_{k}$,
\begin{equation}
B_{k}=\left[\begin{array}{ccccc}
\alpha_{1} & \beta_{1}\\
 & \alpha_{2} & \beta_{2}\\
 &  & \ddots & \ddots\\
 &  &  & \ddots & \beta_{k-1}\\
 &  &  &  & \alpha_{k}
\end{array}\right].\label{eq:bidiag}
\end{equation}
Having in mind relation~\eqref{eq:Rk} and taking $R_{k}=B_{k}$, the vector $v_{k}$ and the entry $\eta_{k}$ in the last column of $B_{k+1}$ are given
by $v_{k}=\beta_{k}e_{k}$, $\eta_{k}=\alpha_{k+1}$, where $e_{k}=[0,\dotsm,0,1]^{T}$ is the $k$th column of the $k\times k$ identity matrix. Hence
\[
\rho_{k}=\|B_{k}z_{k}\|^{2},\qquad\sigma_{k}=\alpha_{k}\beta_{k}e_{k}^{T}z_{k},\qquad\tau_{k}=\beta_{k}^{2}+\alpha_{k+1}^{2}.
\]
Note that the last entry $e_{k}^{T}z_{k}$ of the vector $z_{k}$
is given by $c_{k-1}$ (see \eqref{eq:z}), and, therefore, $\sigma_{k}=\alpha_{k}\beta_{k}c_{k-1}$.
Using the previous results, we are now able to update the entries
$\rho_{k}$, $\sigma_{k}$ and $\tau_{k}$ without storing the vector
$z_{k}$; see Algorithm \ref{alg:bidiagnorm}.

\begin{algorithm}[ht]
\caption{Incremental estimation of $\|B_{k}\|^{2}$}
\label{alg:bidiagnorm}

\begin{algorithmic}[0]

\STATE \textbf{input} entries $\alpha_{k}$ and $\beta_{k}$ of upper
bidiagonal matrices

\STATE $\rho_{1}=\alpha_{1}^{2}\mbox{ }$, $\rho_{1}^{\klein{\max}}=\rho_{1}$, $c_{0}=1$,

\FOR{$k=1,\dots$}

\STATE $\sigma_{k}^{2}=\alpha_{k}^{2}\beta_{k}^{2}c_{k-1}^{2}$,
$\tau_{k}=\beta_{k}^{2}+\alpha_{k+1}^{2}$

\STATE $\chi_{k}^{2}=\left(\rho_{k}-\tau_{k}\right)^{2}+4\sigma_{k}^{2}$

\STATE $c_{k}^{2}=\frac{1}{2}\left(1-\frac{\rho_{k}-\tau_{k}}{\chi_{k}}\right)$

\STATE $\rho_{k+1}=\rho_{k}+\chi_{k}c_{k}^{2}$

\STATE $\rho_{k+1}^{\klein{\max}}=\rho_{k+1}$

\ENDFOR

\end{algorithmic}
\end{algorithm}

In some cases, a better accuracy of the approximations to norms
of matrices is needed. To improve the accuracy,
we need to store $B_{k}$ and $z_{k}$ so that we can run Algorithm~\ref{alg:incremental}, and
construct the approximate maximum right singular vector
\begin{equation}\label{eq:zkp1}
z_{k+1}=\left[\begin{array}{c}
s_{k}z_{k}\\
c_{k}
\end{array}\right]
\end{equation}
of $B_{k+1}$. The vector $z_{k+1}$ can also be seen as an
approximate eigenvector of $B_{k+1}^{T}B_{k+1}$ corresponding to the approximate maximum eigenvalue $\rho_{k+1}$. Hence, one can improve
the vector $z_{k+1}$ using one shifted inverse iteration applied to the
matrix  $B_{k+1}^{T}B_{k+1}$, where $\rho_{k+1}$ is used as a shift; see, e.g., \cite[Section~7.6]{B:GoLo2013}.

In detail, having the $LDL^T$ factorization of the tridiagonal matrix $B_{k+1}^{T}B_{k+1}$, we can easily compute the $LDL^T$
factorization of the matrix $B_{k+1}^{T}B_{k+1}-\rho_{k+1}I$ using the \texttt{dstqds} algorithm  by Parlett and Dhillon \cite{PaDh2000}.
The last factorization can be used to perform one inverse iteration by solving the system
\[
(B_{k+1}^{T}B_{k+1}-\rho_{k+1}I) y = z_{k+1}.
\]
Finally, we can consider the vector $\widehat{z}_{k+1}\equiv y/\|y\|$ and the
scalar  $\widehat{\rho}_{k+1}\equiv \| B_{k+1} \widehat{z}_{k+1}\|^2$
to be new approximations to the maximum right singular vector and to the squared norm of $B_{k+1}$, and
$\widehat{\rho}_{k+1}^{\klein{\max}} \equiv \widehat{\rho}_{k+1}$ to be an improved estimate of the largest
eigenvalue of $B_{k+1}^T B_{k+1}$.

\subsection{Inversions of nonsingular upper bidiagonal matrices}
\label{sec:NormInvB}
Consider a nonsingular bidiagonal matrix $B_{k}$ of the form \eqref{eq:bidiag}, $\alpha_{i}\neq0$. It is well known that the last column $w_{k+1}$ of
the matrix $B_{k+1}^{-1}$ can be expressed in the explicit form
\[
w_{k+1}=\frac{1}{\alpha_{k+1}}\left[\begin{array}{ccccc}
\left(-1\right)^{k}\prod_{i=1}^{k}\frac{\beta_{i}}{\alpha_{i}}, & \dots, & \frac{\beta_{k-1}}{\alpha_{k-1}}\frac{\beta_{k}}{\alpha_{k}}, & -\frac{\beta_{k}}{\alpha_{k}}, & 1\end{array}\right]^{T}.
\]
Hence,
\begin{equation}
B_{k+1}^{-1}=\left[\begin{array}{cc}
B_{k}^{-1} & -w_{k}\frac{\beta_{k}}{\alpha_{k+1}}\\
 & \frac{1}{\alpha_{k+1}}
\end{array}\right],\qquad w_{k+1}=\frac{1}{\alpha_{k+1}}\left[\begin{array}{c}
-w_{k}\beta_{k}\\
1
\end{array}\right],\label{eq:Binv}
\end{equation}
where $w_{k}$ is the last column of the matrix $B_{k}^{-1}$. We
now specialize the idea of the incremental norm estimation presented
in Section \ref{sec:incremental} to the case of matrices $B_{k}^{-1}$,
that is,
\[
R_{k}=B_{k}^{-1},\quad v_{k}=-w_{k}\frac{\beta_{k}}{\alpha_{k+1}},\quad\eta_{k}=\frac{1}{\alpha_{k+1}}.
\]

First, let us find updating formulas for $\|w_{k+1}\|^{2}$ and $B_{k}^{-T}w_{k}$.
From \eqref{eq:Binv} it follows that
\begin{equation}
\|w_{k+1}\|^{2}=\frac{1}{\alpha_{k+1}^{2}}\left(\beta_{k}^{2}\|w_{k}\|^{2}+1\right),\label{eq:up1}
\end{equation}
and
\begin{equation}
B_{k}^{-T}w_{k}=\left[\begin{array}{cc}
B_{k-1}^{-T}\\
-w_{k-1}^{T}\frac{\beta_{k-1}}{\alpha_{k}} & \frac{1}{\alpha_{k}}
\end{array}\right]\left[\begin{array}{c}
-w_{k-1}\frac{\beta_{k-1}}{\alpha_{k}}\\
\frac{1}{\alpha_{k}}
\end{array}\right]=\left[\begin{array}{c}
-\frac{\beta_{k-1}}{\alpha_{k}}\left(B_{k-1}^{-T}w_{k-1}\right)\\
\|w_{k}\|^{2}
\end{array}\right].\label{eq:up2}
\end{equation}
Using the formulas \eqref{eq:up1} and \eqref{eq:up2} we are now
able to update the entries $\sigma_{k}$ and $\tau_{k}$ which are
needed in the process of the incremental norm estimation; see Section~\ref{sec:incremental}.
For $\tau_{k}$ we get
\[
\tau_{k}=\|w_{k+1}\|^{2}=\frac{1}{\alpha_{k+1}^{2}}\left(\beta_{k}^{2}\|w_{k}\|^{2}+1\right)=\frac{1}{\alpha_{k+1}^{2}}\left(\beta_{k}^{2}\tau_{k-1}+1\right),
\]
and for $\sigma_{k}$,
\begin{eqnarray*}
\sigma_{k} & = & v_{k}^{T}R_{k}z_{k}=-\frac{\beta_{k}}{\alpha_{k+1}}z_{k}B_{k}^{-T}w_{k}\\
 & = & -\frac{\beta_{k}}{\alpha_{k+1}}\left[\begin{array}{c}
s_{k-1}z_{k-1}\\
c_{k-1}
\end{array}\right]^{T}\left[\begin{array}{c}
-\frac{\beta_{k-1}}{\alpha_{k}}\left(B_{k-1}^{-T}w_{k-1}\right)\\
\|w_{k}\|^{2}
\end{array}\right]\\
 & = & -\frac{\beta_{k}}{\alpha_{k+1}}\left(s_{k-1}\left[-\frac{\beta_{k-1}}{\alpha_{k}}z_{k-1}^{T}B_{k-1}^{-T}w_{k-1}\right]+c_{k-1}\|w_{k}\|^{2}\right)\\
 & = & -\frac{\beta_{k}}{\alpha_{k+1}}\left(s_{k-1}\sigma_{k-1}+c_{k-1}\tau_{k-1}\right).
\end{eqnarray*}
The initial values
\[
 \rho_{1}=\frac{1}{\alpha_{1}^{2}},\quad\tau_{0}=\frac{1}{\alpha_{1}^{2}},\quad\sigma_{0}=0,\quad s_{0}=0,\quad c_{0}=1,
\]
lead to the $2\times2$ matrix
\[
\left[\begin{array}{cc}
\rho_{1} & \sigma_{1}\\
\sigma_{1} & \tau_{1}
\end{array}\right]=\left[\begin{array}{cc}
\frac{1}{\alpha_{1}^{2}} & -\frac{\beta_{1}}{\alpha_{2}\alpha_{1}^{2}}\\
-\frac{\beta_{1}}{\alpha_{2}\alpha_{1}^{2}} & \frac{1}{\alpha_{2}^{2}}+\left(\frac{\beta_{1}}{\alpha_{2}\alpha_{1}}\right)^{2}
\end{array}\right]=B_{2}^{-T}B_{2}^{-1},
\]
so that $\rho_{2}=\left\Vert B_{2}^{-1}\right\Vert ^{2}$. The results
are summarized in Algorithm~\ref{alg:Binv}.

\begin{algorithm}[ht]
\caption{Incremental estimation of $\|B_{k}^{-1}\|^{2}$}
\label{alg:Binv}

\begin{algorithmic}[0]

\STATE \textbf{input} entries $\alpha_{k}$ and $\beta_{k}$ of upper
bidiagonal matrices

\STATE $\rho_{1}=\alpha_{1}^{-2}$, $\rho_{1}^{\klein{\min}}=\alpha_{1}^{2}$,
$\tau_{0}=\rho_{1}$, $\sigma_{0}=0$,
$s_{0}=0$, $c_{0}=1$

\FOR{$k=1,\dots$}

\STATE $\sigma_{k}=-\frac{\beta_{k}}{\alpha_{k+1}}\left(s_{k-1}\sigma_{k-1}+c_{k-1}\tau_{k-1}\right)$,

\STATE $\tau_{k}=\frac{1}{\alpha_{k+1}^{2}}\left(\beta_{k}^{2}\tau_{k-1}+1\right)$

\STATE $\chi_{k}^{2}=\left(\rho_{k}-\tau_{k}\right)^{2}+4\sigma_{k}^{2}$

\STATE $c_{k}^{2}=\frac{1}{2}\left(1-\frac{\rho_{k}-\tau_{k}}{\chi_{k}}\right)$

\STATE $\rho_{k+1}=\rho_{k}+\chi_{k}c_{k}^{2}$

\STATE $s_{k}=\sqrt{1-c_{k}^{2}}$, $c_{k}=|c_{k}|\,\mathrm{sign}(\sigma_{k})$

\STATE $\rho_{k+1}^{\klein{\min}}=\rho_{k+1}^{-1}$
\ENDFOR

\end{algorithmic}
\end{algorithm}

Similarly as in the previous section, we can improve the accuracy of the approximations of norms of inverses of matrices by one shifted inverse
iteration. To do so, we need to store $B_{k}$,  $z_{k}$, and also the vector $B_{k}^{-T}w_{k}$ (to compute $\sigma_k$) which can be updated using the
formula \eqref{eq:up2}. Then, as in \eqref{eq:zkp1}, we can construct the approximate maximum right singular vector $z_{k+1}$ of $B_{k+1}^{-1}$. The
vector $z_{k+1}$ can be seen as an approximate eigenvector of the matrix $B_{k+1}^{-T}B_{k+1}^{-1}$, or, as an approximate
 eigenvector of the matrix $B_{k+1}B_{k+1}^{T}$,
\[
B_{k+1}^{-T}B_{k+1}^{-1} z_{k+1} \approx \rho_{k+1} z_{k+1},\qquad
\rho_{k+1}^{-1} z_{k+1} \approx
B_{k+1}B_{k+1}^{T} z_{k+1}.
\]
The accuracy  of the vector $z_{k+1}$ can now be improved by one shifted inverse iteration applied to the
matrix  $B_{k+1}B_{k+1}^{T}$, where $\rho_{k+1}^{-1}$ is used as a shift.

In detail, we can easily get
the  $UDU^T$ factorization ($U$ is upper bidiagonal) of the tridiagonal matrix $B_{k+1}B_{k+1}^T$. Using a straightforward modification of the
\texttt{dstqds} algorithm, the
  $UDU^T$ factorization of the matrix $B_{k+1}B_{k+1}^{T}-\rho_{k+1}^{-1}I$ can be computed
  and used to solve the system
\[
 (B_{k+1}B_{k+1}^{T}-\rho_{k+1}^{-1}I) y = z_{k+1}.
\]
The modification of the \texttt{dstqds} algorithm consists in the unitary transformation of the problem for the $UDU^T$ factorization to
the problem with $LDL^T$ factorization, using the backward identity matrix.
Finally, one can consider the vector $\widehat{z}_{k+1}\equiv y/\|y\|$ and the scalar  $\widehat{\rho}_{k+1}\equiv \| B_{k+1}^{-1} \widehat{z}_{k+1}\|^{2}$
to be new approximations to the maximum right singular vector and to $\|B_{k+1}^{-1}\|^2$, and $\widehat{\rho}_{k+1}^{\klein{\min}} \equiv \widehat{\rho}_{k+1}^{-1}$ to be an improved estimate of the smallest eigenvalues of $B_{k+1}^TB_{k+1}$.

\subsection{CG and approximations of the extreme Ritz values}
\label{sec:CGetreme}

The results of the previous sections can be applied to the upper bidiagonal matrices $B_{k}=L_{k}^{T}$ that are computed in CG, i.e.,
\[
\alpha_{j}=\frac{1}{\sqrt{\gamma_{j-1}}},\ j=1,\dots,k,\qquad\beta_{j}=\sqrt{\frac{\delta_{j}}{\gamma_{j-1}}},\ j=1,\dots,k-1,
\]
to approximate the smallest and largest eigenvalues of $T_{k}$; see \eqref{eq:why}. In particular, after substitution we obtain in Algorithm
\ref{alg:bidiagnorm},
\[
\sigma_{k}^{2}=\frac{\delta_{k}}{\gamma_{k-1}^{2}}c_{k-1}^{2},\qquad\tau_{k}=\frac{1}{\gamma_{j-1}}+\frac{\delta_{k}}{\gamma_{j-1}},
\]
and in Algorithm \ref{alg:Binv},
\begin{equation}\label{eq:coeff}
\sigma_{k}=-\sqrt{\gamma_{k}\frac{\delta_{k}}{\gamma_{k-1}}}\left(s_{k-1}\sigma_{k-1}+c_{k-1}\tau_{k-1}\right),\qquad\tau_{k}=\gamma_{k}\left(\delta_{k}\frac{\tau_{k-1}}{\gamma_{k-1}}+1\right).
\end{equation}
Moreover, for $\tau_{k}$ in Algorithm \ref{alg:Binv} it holds that
\[
\frac{\tau_{k}}{\gamma_{k}}=\left[1+\delta_{k}(1+\delta_{k-1}(1+\dots+\delta_{2}(1+\delta_{1})\dots))\right]=\frac{\left\Vert p_{k}\right\Vert ^{2}}{\left\Vert r_{k}\right\Vert ^{2}}.
\]

\section{Approximation of the Gauss-Radau upper bound}
 \label{sec:radau}
The previous section provides a cheap tool to approximate the Gauss-Radau upper bound without having an a priori information about the smallest eigenvalue of the (preconditioned) system matrix. In particular, to approximate the Gauss-Radau upper bound one can use the new upper bound \eqref{eq:newbound}. Instead of $\mu$ which should closely approximate the smallest eigenvalue from below,  one can use the updated approximation $\mu_k\equiv \rho_{k}^{\klein{\min}}$ to the smallest Ritz value; see Algorithm~\ref{alg:Binv} and  Section~\ref{sec:CGetreme}. Since the bound \eqref{eq:newbound} is not sensitive to the choice of $\mu$, the approximative bound \eqref{eq:newbound} which uses $\mu_k$ will be close to the bound \eqref{eq:newbound} for  $\mu=\lambda_{\min}$ whenever $\mu_k \approx\lambda_{\min}$. Moreover, as we have seen in Section~\ref{sec:bcsstk01}, the bound \eqref{eq:newbound} is often a good approximation to the Gauss-Radau upper bound, in particular if $\mu$ approximates the smallest eigenvalue only roughly, say to $1$ or $2$ valid digits. In summary, when we do not have an a priori information about the smallest eigenvalue of the (preconditioned) system matrix, we suggest to approximate the Gauss-Radau upper bound $u_k^{\smu}$, see \eqref{eq:denote}, using an approximate upper bound
\begin{equation}\label{eq:approx}
\pmb{u}_k^{\scriptscriptstyle (\mu_k)} = 	\frac{\|r_{k}\|}{\sqrt{\mu_k}}\frac{\|r_{k}\|}{\|p_{k}\|}
\end{equation}
where $\mu_k = \rho_{k}^{\klein{\min}}$ is updated at each iteration as
in Algorithm \ref{alg:Binv}, with $\sigma_{k}$ and $\tau_{k}$ computed directly from the CG coefficients using \eqref{eq:coeff}. The algorithm for updating $\mu_k$ starts with $\rho_{1}=\gamma_0$, $\mu_1=\gamma_0^{-1}$,
$\tau_{0}=\rho_{1}$, $\sigma_{0}=0$,
$s_{0}=0$, $c_{0}=1$. Note that it does not make too much sense to 
use inverse iterations to improve the quality of the approximation
of the smallest Ritz value. A more accurate approximation to the smallest Ritz value does not improve the bound \eqref{eq:approx} significantly.

\section{Approximation of the normwise backward error}
 \label{sec:be}
In \cite{PaSa1982,ArDuRu1992}, backward error perturbation theory was used to derive a family of stopping criteria for iterative methods. In particular, given $\widetilde x$, one can ask what are the norms of the smallest perturbations $\Delta A$ of $A$ and $\Delta b$ of $b$
measured in the relative sense such that
the approximate solution $\widetilde x$ represents the exact solution of
the perturbed system
\[
(A+\Delta A)\, \widetilde x =b+\Delta b\,.
\]
In other words, we are interested in the quantity
\[
\eta = \min \left\{\delta:\ (A+\Delta A)\, \widetilde x=b+\Delta b,\ \| \Delta A\| \leq \delta \| A\|,\ \| \Delta b\| \leq \delta \| b\| \right\}\,.
\]
It was shown by Rigal and Gaches \cite{RiGa1967} that this quantity, called {\em the normwise backward error}, is given by
\begin{equation}
\eta = \frac{\|\widetilde r\|}{\|A\|\|\widetilde x \|+\|b\|}\,. \label{eq:bck}
\end{equation}
where $\widetilde r = b-A\widetilde x$. This approach can be generalized, see \cite{PaSa1982,ArDuRu1992}, in order to quantify levels of confidence
in $A$ and $b$. The normwise backward
error is, as a base for stopping criteria, frequently recommended
in the numerical analysis literature, see, e.g. \cite{B:Hi1996,Ba1994}.

When solving a linear system with CG, the norms of vectors
$\widetilde x=x_k$ and $\widetilde r=r_k$ are easily computable, and
 $\|A\|$ can be approximated from below using Algorithm~\ref{alg:bidiagnorm}; see also Section~\ref{sec:CGetreme}. Hence, we can efficiently compute an approximate upper bound on the normwise backward error \eqref{eq:bck} in CG. In the following subsection we show that if $x_0=0$, then $\|x_k\|$ can be approximated cheaply in an incremental way.

\subsection{A cheap approximation of $\|x_k\|$ in CG}
\label{sec:normx}
If $x_0=0$, then the CG approximate solution $x_k$ can be expressed as
\[
x_{k}=\|r_{0}\|V_{k}T_{k}^{-1}e_{1},\quad\mbox{and} \quad
\|x_{k}\|^{2}=\|r_{0}\|^{2}e_{1}^{T}T_{k}^{-1}V_{k}^{T}V_{k}T_{k}^{-1}e_{1}.
\]
Using the {\em global orthogonality} among the Lanczos vectors we obtain
\begin{equation}\label{eq:global}
\|x_{k}\|^{2}=\|r_{0}\|^{2}e_{1}^{T}T_{k}^{-2}e_{1}.
\end{equation}
Note that in finite precision arithmetic, the orthogonality is usually quickly lost. However, we observed in numerical experiments (see Section~\ref{sec:exp}) that despite the loss or orthogonality, the quantity
\begin{equation}\label{eq:xik}
\xi_k \equiv \|r_{0}\|^{2}e_{1}^{T}T_{k}^{-2}e_{1}
\end{equation}
still approximates $\|x_{k}\|^{2}$ very accurately. In the following lemma we suggest an algorithm to efficiently compute $\xi_k$ at a negligible cost.

\begin{lemma}\label{lem:norm} With the notation introduced in Section \ref{sec:CG}, it holds that
\[
\xi_k=\sum_{j=0}^{k-1}\|r_{j}\|^{-2}\left(\sum_{i=j}^{k-1}\psi_{i}\right)^{2},\qquad \psi_{i}=\gamma_{i}\|r_{i}\|^{2},
\]
and $\xi_{k+1}$,
 $k=0,1,2,\dots$, can be computed using the recurrences
\begin{eqnarray}
\vartheta_{k+1} & = & \vartheta_{k}+\gamma_{k}\phi_{k}^{-1},
\label{eq:nx1}\\
\xi_{k+1} & = & \xi_{k}+\psi_{k}\left(\vartheta_{k+1}+\vartheta_{k}\right),
\label{eq:nx2}
\end{eqnarray}
where $\vartheta_{0}=0$, $\xi_{0}=0$, and  $\phi_{k}$ can be updated using \eqref{eq:phi}.
\end{lemma}

\begin{proof}
It holds that
\[
\xi_{k}=\|r_{0}\|^{2}e_{1}^{T}T_{k}^{-2}e_{1}=\|\|r_{0}\|L_{k}^{-T}L_{k}^{-1}e_{1}\|^{2}\equiv\|y\|^{2}
\]
where $y=[y_{1},\dots,y_{k}]^{T}$ solves the system $L_{k}^{T}L_{k}y=\|r_{0}\|e_{1}.$
Using the bidiagonal structure of $L_{k}$ we get in a straightforward
way that
\[
y_{j}=\left(-1\right)^{j+1}\frac{1}{\|r_{j-1}\|}\left(\sum_{i=j-1}^{k-1}\psi_{i}\right),\qquad j=1,\dots,k,
\]
and, therefore,
\[
\xi_{k}=\|y\|^{2}=\sum_{j=0}^{k-1}\frac{\left(\sum_{i=j}^{k-1}\psi_{i}\right)^{2}}{\|r_{j}\|^{2}}.
\]

It remains to find a way how to compute $\xi_{k}$ in an efficient
way. In other words, knowing $\xi_{k}$ and $\psi_{k}$, we would
like to express $\xi_{k+1}.$ It holds that
\begin{eqnarray*}
\xi_{k+1}=\sum_{j=0}^{k}\frac{\left(\sum_{i=j}^{k}\psi_{i}\right)^{2}}{\|r_{j}\|^{2}} & = &  \frac{\psi_{k}^{2}}{\|r_{k}\|^{2}}+\sum_{j=0}^{k-1}\frac{\psi_{k}^{2}}{\|r_{j}\|^{2}}+\sum_{j=0}^{k-1}\frac{2\psi_{k}\sum_{i=j}^{k-1}\psi_{i}}{\|r_{j}\|^{2}}+\xi_{k}\\
 & = & \psi_{k}\left(\sum_{j=0}^{k}\frac{\sum_{i=j}^{k}\psi_{i}}{\|r_{j}\|^{2}}+\sum_{j=0}^{k-1}\frac{\sum_{i=j}^{k-1}\psi_{i}}{\|r_{j}\|^{2}}\right)+\xi_{k}\\
 &=&\xi_{k}+\psi_{k}\left(\vartheta_{k+1}+\vartheta_{k}\right)
\end{eqnarray*}
where
\[
\vartheta_{k}\equiv\sum_{j=0}^{k-1}\frac{\sum_{i=j}^{k-1}\psi_{i}}{\|r_{j}\|^{2}}.
\]

Let us find an updating formula for $\vartheta_{k+1}$. We have,
\begin{eqnarray*}
\vartheta_{k+1} & = & \gamma_{k}+\sum_{j=0}^{k-1}\frac{\sum_{i=j}^{k-1}\psi_{i}+\psi_{k}}{\|r_{j}\|^{2}}=\gamma_{k}+\vartheta_{k}+\psi_{k}\sum_{j=0}^{k-1}\frac{1}{\|r_{j}\|^{2}}\\
 & = & \gamma_{k}+\vartheta_{k}+\gamma_{k}\frac{\|r_{k}\|^{2}}{\|r_{k-1}\|^{2}}\left(\|r_{k-1}\|^{2}\sum_{j=0}^{k-1}\|r_{j}\|^{-2}\right)\\
 & = & \gamma_{k}+\vartheta_{k}+\gamma_{k}\delta_{k}\phi_{k-1}^{-1}\\
 & = & \vartheta_{k}+\gamma_{k}\phi_{k}^{-1};
\end{eqnarray*}
see \eqref{eq:phi} and \eqref{eq:minres}.
\end{proof}

Lemma~\ref{lem:norm} shows how to cheaply approximate $\|x_k\|$ in CG under the assumption $x_0=0$. 
If $x_0\neq 0$, then $x_k = x_0 +
\|r_{0}\| V_{k}T_{k}^{-1}e_{1}$ and $\xi_k$ can be seen as an approximation to $\|x_k-x_0\|^2$. By a simple algebraic manipulation we can express $\|x_k\|^2$ as
\begin{equation}\label{eq:term}
\|x_k\|^2 = \|x_0\|^2 + 2 \|r_{0}\| x_0^T V_{k}T_{k}^{-1}e_{1} + \xi_k.
\end{equation}
The term $(V_{k}^T x_0)^TT_{k}^{-1}e_{1}$ can be evaluated incrementally, without storing the Lanczos vectors. However, this requires the computation
of one additional inner product per iteration. While in CG, it is then better to compute directly $\|x_k\|^2$, the term \eqref{eq:term} can still be useful
in PCG where norms of preconditioned approximations can be of interest when approximating the normwise backward error which corresponds to the preconditioned system.

\begin{algorithm}[ht]
\caption{Preconditioned Conjugate gradients (PCG)} \label{alg:pcg}
\begin{algorithmic}[0]
\STATE \textbf{input} $A$, $b$, $x_{0}$, $M$,
\STATE $r_{0}=b-Ax_{0}$, {\tt solve} $M z_{0}= r_{0}$ {\tt to get} $z_0$, $p_{0}=z_{0}$
\FOR{$k=1,\dots$ {\tt until convergence}}
\STATE $\widehat{\gamma}_{k-1}=\frac{z_{k-1}^{T}r_{k-1}}{p_{k-1}^{T}Ap_{k-1}}$
\STATE $x_{k}=x_{k-1}+\widehat{\gamma}_{k-1}p_{k-1}$ \STATE $r_{k}=r_{k-1}-\widehat{\gamma}_{k-1}Ap_{k-1}$
\STATE {\tt solve} $M z_{k}= r_{k}$ {\tt to get} $z_k$
\STATE $\widehat{\delta}_{k}=\frac{z_{k}^{T}r_{k}}{z_{k-1}^{T}r_{k-1}}$ \STATE $p_{k}=z_{k}+\widehat{\delta}_{k}p_{k-1}$
 \ENDFOR
\end{algorithmic}%
\end{algorithm}
\subsection{Normwise backward error in PCG}\label{sec:beerPCG}
Given a symmetric positive definite matrix $M=LL^T$ we can formally think about preconditioned CG (see Algorithm~\ref{alg:pcg}) as CG applied to the
modified system
\begin{equation}
 \label{eq:prec}
 \underbrace{L^{-1}A L^{-T}}_{\widehat{A}} \underbrace{L^T x}_{\widehat{x}} = \underbrace{L^{-1} b}_{\widehat{b}}.
\end{equation}
Moreover, a change of variable is used to go back to the original variable $x$ and the original residual $r$ in such a way that the only
preconditioning matrix which is involved is $M$ or its inverse, and not $L$ which may be unknown. Using the techniques presented in
Sections~\ref{sec:Ritz_val} and \ref{sec:normx} we can approximate the normwise backward error for the preconditioned system \eqref{eq:prec},
\begin{equation}\label{eq:pberr}
 \widetilde \eta = \frac{\Vert\widetilde r\Vert}{\Vert \widehat{A} \Vert\,\Vert \widetilde x\Vert + \Vert \widehat{b}\Vert},
\end{equation}
where $\widetilde x$ is a given approximation and $\widetilde r = \widehat{b} - \widehat{A} \widetilde x$. In particular, in PCG we are interested in
$\widetilde x=L^T x_k$, $\widetilde r=L^{-1}r_k$, so that
\begin{equation}\label{eq:precond}
    \Vert\widetilde x\Vert^2 = \Vert x_k \Vert_M^2,\quad
    \Vert\widetilde r\Vert^2 = z_k^Tr_k = \Vert r_k \Vert_{M^{-1}}^2,\quad
    \Vert\widehat b\Vert^2 = \Vert b \Vert_{M^{-1}}^2.
\end{equation}
The norm of the preconditioned matrix $\Vert \widehat{A}\Vert$ can be approximated from the PCG coefficients $\widehat{\gamma}_k$ and $\widehat{\delta}_k$ using techniques developed
in Section~\ref{sec:Ritz_val}, and the norm of the preconditioned approximation $\Vert L^{T}x_k \Vert = (x_k^T M x_k)^{1/2} = \Vert x_k \Vert_M $ can be approximated using Lemma~\ref{lem:norm}. The other quantities are available in PCG.

We know that $\widetilde x$ is the exact solution of a perturbed problem $(\widehat A +\Delta \widehat A)\widetilde x = \widehat b + \Delta \widehat b$, where the relative sizes of $\Delta \widehat A$ and $\Delta \widehat b$ are bounded by $\widetilde \eta$. Hence, $x_k$ is
the exact solution of the perturbed system
\[
(A + \widehat\Delta_A) x_k =  b + \widehat\Delta_b,\qquad
\widehat\Delta_A  \equiv L(\Delta \widehat A )L^{T},\qquad \widehat\Delta_b \equiv L (\Delta \widehat b)\,.
\]
Since the relative sizes of  $\Delta \widehat A$ and $\Delta \widehat b$ are bounded by $\widetilde \eta$, it holds that $\Vert \widehat\Delta_A \Vert / \Vert A \Vert \leq
\kappa(M)\widetilde \eta$ and $\Vert \widehat\Delta_b \Vert / \Vert b \Vert \leq \kappa(M)^{1/2}\widetilde \eta.$

Nevertheless, the question of which backward error makes more sense in a given problem remains. The quantity $\eta$ in \eqref{eq:bck} tells us how well we have solved
the original system whence $\widetilde\eta$ in \eqref{eq:pberr} tells us how well we have solved the preconditioned system. We did not find any
discussion of this issue in the literature.

\section{Numerical experiments}
 \label{sec:exp}
Numerical experiments are divided into two parts. In the first part we
demonstrate the quality of our estimates approximating the extreme Ritz values and the norms of approximate solutions during the CG computations. In the second part we use
these estimates to approximate characteristics of our interest, i.e., the Gauss-Radau upper bound for the $A$-norm of the error and the normwise backward error.
The experiments are performed in Matlab 9.2 (R2017a).

We consider four systems of linear equations. The first one with the system matrix {\tt bcsstk01} has already been described in Section~\ref{sec:bcsstk01}. For this system, the influence of finite precision arithmetic to CG computations is substantial; orthogonality is quickly lost and convergence is significantly delayed. Hence, one can test whether our techniques work also under these circumstances which are quite realistic during practical computations.

The second system arises after discretizing the diffusion equation
$$
-\mathrm{div}(\lambda(x,y)\nabla u)=f\ \ \mbox{in}\ \ \Omega=(0,1)^2,\quad
u|_{\partial \Omega}=0,
$$
with the diffusion coefficient
$$
    \lambda(x,y) = \frac{1}{2 + 1.8 \sin(10x)}\cdot
    \frac{1}{2 + 1.8  \sin(10\,y)}\,.
$$
The PDE is discretized using standard finite differences with a five-point scheme on a $60 \times 60$ mesh so that the system matrix {\tt Pb26} has the moderate dimension $3600$;
for more details see \cite[Section~9.2, p. 313]{B:Me2006}.
Note that $\mathrm{nnz}(A)=17760$ and $\kappa(A) \approx 7.54\times 10^4$.
The right hand side $b$ is a random vector normalized to have a unit norm.
The starting vector is $x_0=0$. In the experiments, the system is solved without preconditioning.

The third linear system {\tt Pres\_Poisson} from the 
SuiteSparse Matrix Collection arises in problems of computational fluid dynamics. The matrix size is $n=14822$, $\mathrm{nnz}(A)=715804$, $\kappa(A)\approx 2.04 \times 10^6$, the right hand side $b$ is provided with the matrix. The starting vector is $x_0=0$. We use incomplete Cholesky factorization with zero-fill
as a preconditioner; see, e.g., \cite[Section~11.5.8]{B:GoLo2013}.

Finally, the last system matrix {\tt s3dkt3m2} is of order $n = 90449$ and $\kappa(A) \approx 3.6 \times 10^{11}$. It can be downloaded from the
CYLSHELL collection in the Matrix Market library, which
contains matrices that represent low-order finite-element discretizations of a shell element test, the pinched
cylinder. Only the last element of the right-hand side vector $b$ is nonzero, which corresponds to the given
physical problem; for more details see \cite{Ko1999} and the references therein. The factor $L$ in the preconditioner
$M = LL^T$ is determined by the incomplete Cholesky (ichol) factorization with threshold dropping,
{\tt type = 'ict'}, {\tt droptol = 1e-5}, and with the global diagonal shift {\tt diagcomp = 1e-2}. Note that here
$\mathrm{nnz}(A)=3686223$ and $\mathrm{nnz}(L)=6541916$. When used in  experiments, the smallest eigenvalue of the preconditioned matrix 
was computed as the smallest Ritz value at the iteration $k=3500$ for which the ultimate level of accuracy of the $A$-norm of the error was already reached.

\subsection{Approximations to the extreme Ritz values and to $\|x_k\|$} \label{sec:extreme}
It is sometimes difficult to know beforehand good approximations of the smallest and largest eigenvalues of $A$. Since CG is equivalent to the
Lanczos algorithm, estimates of the smallest and largest eigenvalues can be computed during CG iterations via approximating the smallest and largest Ritz value.
In Algorithms~\ref{alg:bidiagnorm} and \ref{alg:Binv}, and in
Section~\ref{sec:CGetreme} we formulated a very cheap way of approximating the extreme Ritz values at a negligible cost of a few scalar operations per iteration.
Moreover, the estimates can be improved when updating the $LDL^T$ factorization of the tridiagonal matrix $T_k$ and performing one shifted inverse iteration; see Sections~\ref{sec:NormB} and \ref{sec:NormInvB}.

Note that an adaptive algorithm for approximating
the smallest eigenvalue was also proposed in \cite{Me1997}, with the aim
to get the parameter $\mu$ for computing the
Gauss-Radau bound. The user was required to provide an initial lower bound
for $\lambda_{\min}(A)$. Then, during the CG iterations
the smallest Ritz values were computed using a fixed number of inverse iterations. When the smallest Ritz value was considered to be converged, the value of $\mu$ was changed to the converged value. However, this required solving several tridiagonal linear systems at every CG iteration and the size of these linear systems was increasing with the CG iterations.
Therefore, we can do now something better with our new cheap estimates, as well as with the improved estimates which require solving of just one linear system per iteration.

Let us first describe the meaning of curves in Figures~\ref{fig3}-\ref{fig6}. The left and right parts of the figures correspond to approximations of the largest and smallest eigenvalue, respectively. Denote by $\theta_1^{\klein{(k)}},\dots,\theta_k^{\klein{(k)}}$
the eigenvalues of $T_k$, i.e., the Ritz values, sorted in nondecreasing
order, which we compute 
using the Matlab command {\tt eig}. We plot the convergence history of the relative distance of the largest or smallest Ritz value to the largest or smallest eigenvalue of $A$ respectively, i.e., the quantities
\[
\frac{|\lambda_{\max}(A)-\theta_k^{\klein{(k)}}|}{\lambda_{\max}(A)},\qquad \frac{|\lambda_{\min}(A)-\theta_1^{\klein{(k)}}|}{\lambda_{\min}(A)},
\]
as a dash-dotted curve.
The dashed and dotted curves are related to the relative accuracy of the estimates of the largest or smallest Ritz value,
\[
\frac{|\theta_k^{\klein{(k)}} - \mathrm{est}_k^{\klein{\max}}|}{\theta_k^{\klein{(k)}}},\qquad \frac{|\theta_1^{\klein{(k)}} - \mathrm{est}_k^{\klein{\min}}|}{\theta_1^{\klein{(k)}}},
\]
where $\mathrm{est}_k^{\klein{\max}}$ stands for $\rho_k^{\klein{\max}}$ or $\widehat{\rho}_k^{\klein{\max}}$, and $\mathrm{est}_k^{\klein{\min}}$ stands for
$\rho_k^{\klein{\min}}$ or $\widehat{\rho}_k^{\klein{\min}}$.
In particular, the dashed curves correspond to the relative accuracy of the cheap estimates  $\rho_k^{\klein{\max}}$ and $\rho_k^{\klein{\min}}$
computed by Algorithms \ref{alg:bidiagnorm} and \ref{alg:Binv} respectively,
while the dotted curve corresponds
to the relative accuracy of the improved estimates  $\widehat{\rho}_k^{\klein{\max}}$ and $\widehat{\rho}_k^{\klein{\min}}$, described in Section \ref{sec:NormB} and \ref{sec:NormInvB}.
Finally, the relative distances of the cheap estimates 
${\rho}_k^{\klein{\max}}$ and ${\rho}_k^{\klein{\min}}$
to the largest and smallest eigenvalues, i.e.,
\[
\frac{|\lambda_{\max}(A)-{\rho}_k^{\klein{\max}}|}{\lambda_{\max}(A)},\qquad \frac{|\lambda_{\min}(A)-{\rho}_k^{\klein{\min}}|}{\lambda_{\min}(A)},
\]
are plotted as a solid curve. Note that
$
\lambda_{\min}(A) < \theta_1^{\klein{(k)}} \leq \mathrm{est}_k^{\klein{\min}}$
and
$
\mathrm{est}_k^{\klein{\max}} \leq  \theta_k^{\klein{(k)}}< \lambda_{\max}(A)
$.

\begin{figure}
\centering
 \includegraphics[width=5.8cm]{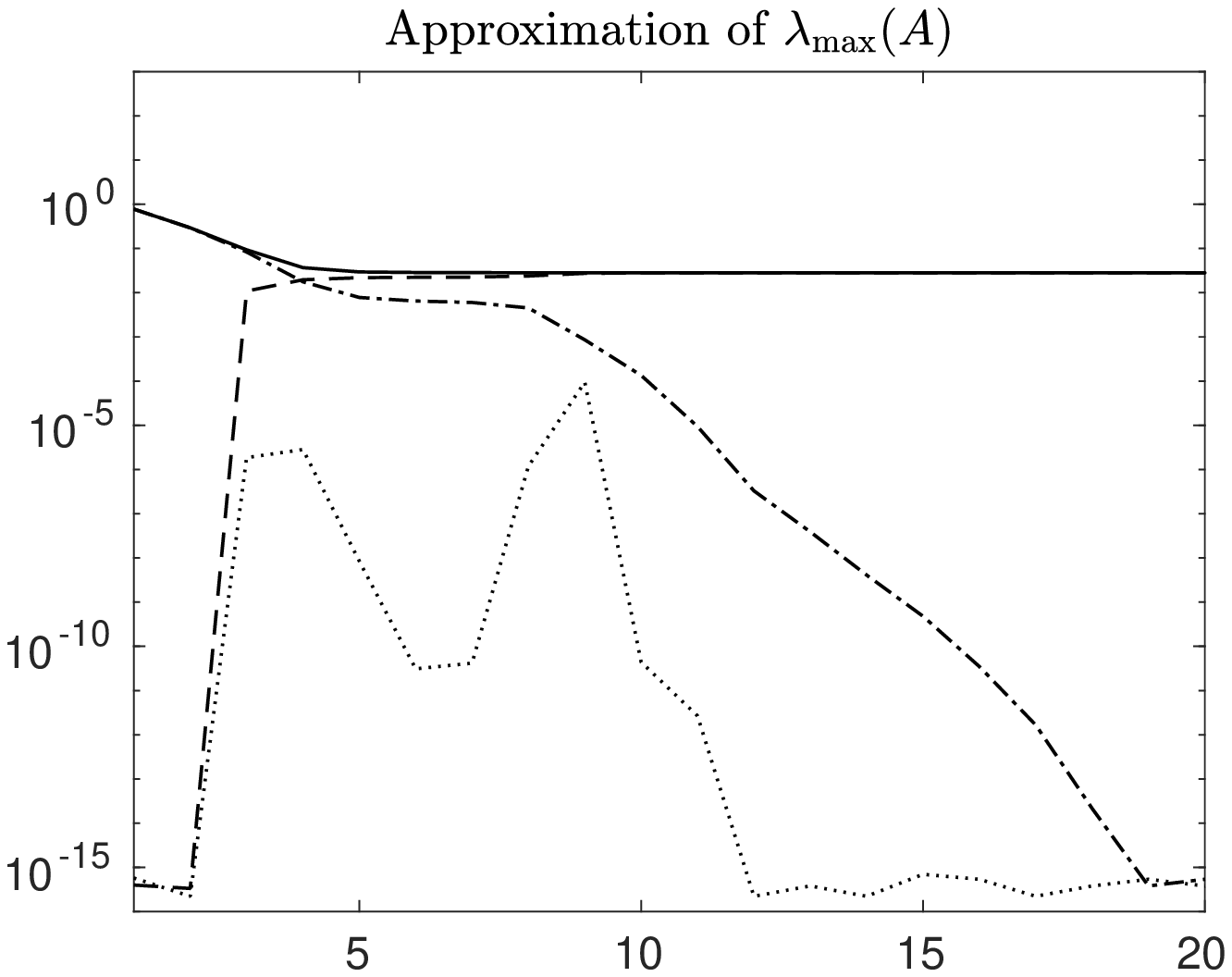}~~~
 \includegraphics[width=5.8cm]{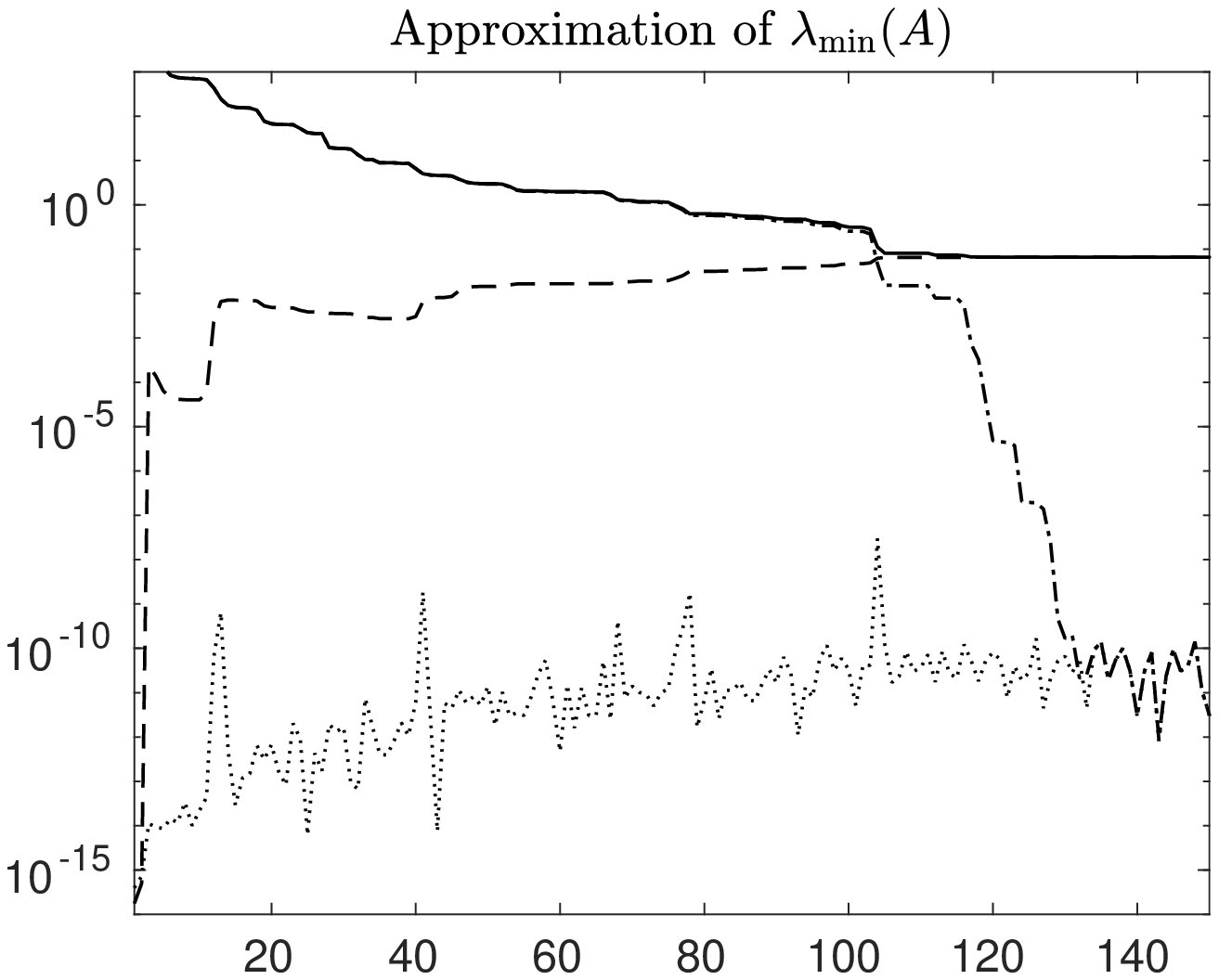}
\caption{Approximating the extreme Ritz values for the system {\tt bcsstk01}.} \label{fig3}
\end{figure}
\begin{figure}[!htbp]
\centering
 \includegraphics[width=5.8cm]{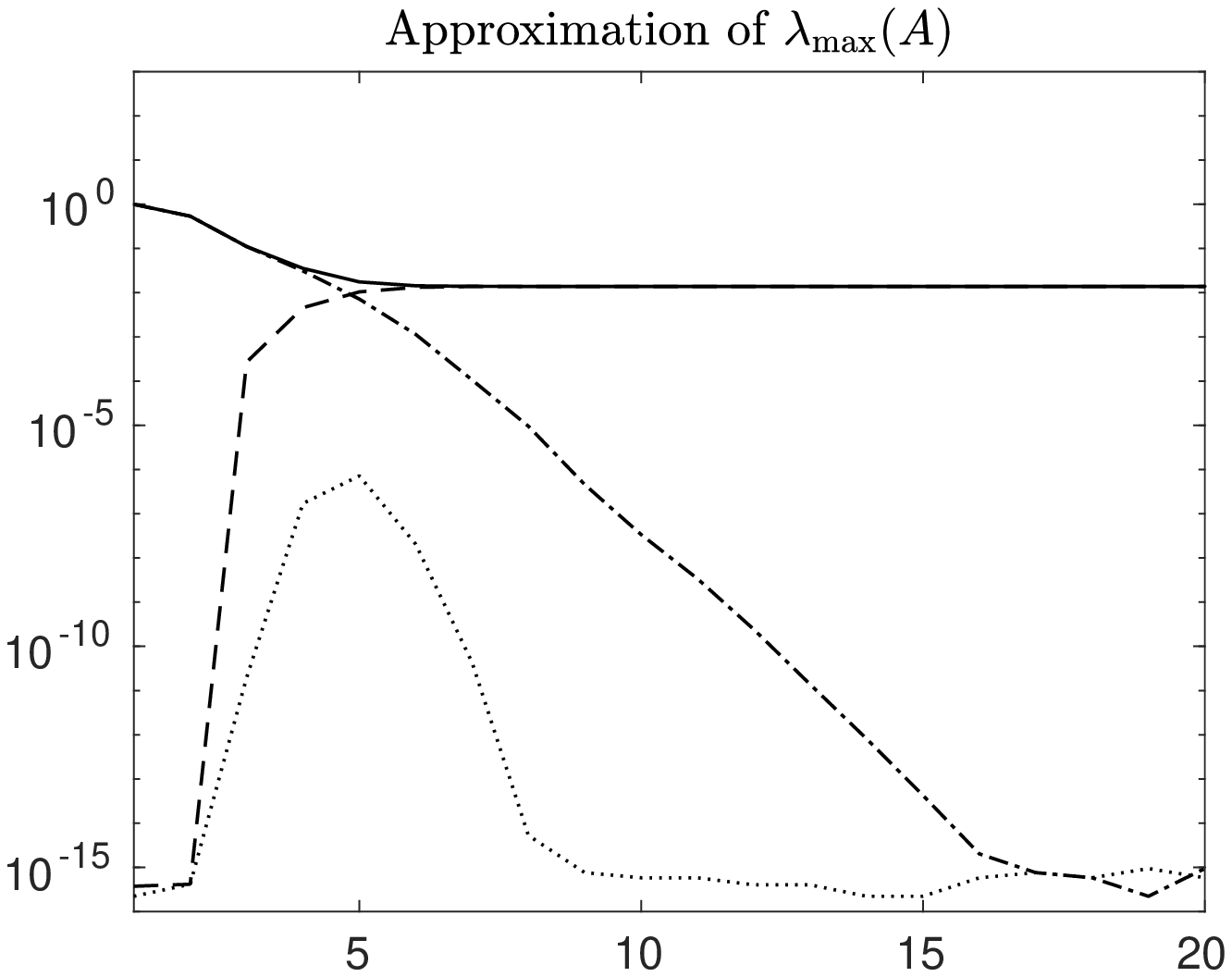}~~~
 \includegraphics[width=5.8cm]{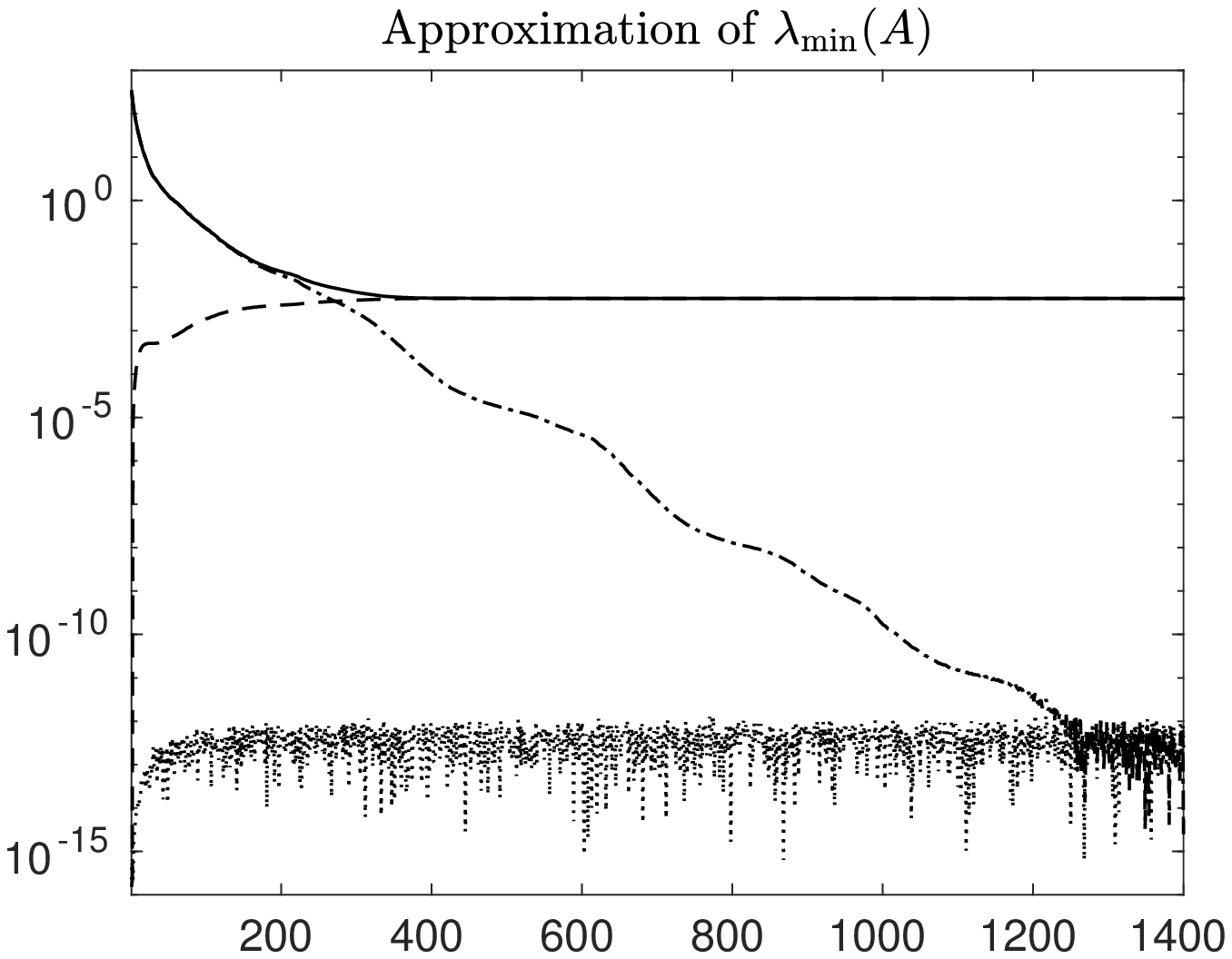}
\caption{Approximating the extreme Ritz values for the system {\tt Pb26}.} \label{fig4}
\end{figure}

In Figures \ref{fig3}-\ref{fig4} we can observe that if CG is applied to an {\em unpreconditioned system}, the largest Ritz values $\theta_k^{\klein{(k)}}$ converge to $\lambda_{\max}(A)$ after a few iterations of CG (dash-dotted curve in the left part), while
 convergence of the smallest Ritz values $\theta_1^{\klein{(k)}}$ to $\lambda_{\min}(A)$
(dash-dotted curve in the right part) is often delayed, and it is usually related to the convergence of the $A$-norm of the error.

In a few initial iterations, the cheap estimates $\rho_k^{\klein{\max}}$ and $\rho_k^{\klein{\min}}$ (dashed curves) approximate the corresponding Ritz values
with a very high accuracy (in theory, the estimates 
agree with the exact Ritz values in iterations 1 and 2). However, in later iterations, their relative accuracy stagnates on the level of $10^{-1}$ or $10^{-2}$. In other words, 
the estimates agree with the corresponding Ritz values to 1 or 2 valid digits.
 Since the extreme Ritz values approximate the extreme eigenvalues of $A$,  the estimates also approximate these eigenvalues.
We can observe that if an extreme Ritz value has converged, then
its cheap estimate approximates the corresponding extreme eigenvalue to
1 or 2 valid digits (solid curve). Note that in most applications, this would be a sufficient accuracy.
The dotted curves show the relative accuracy of the improved estimates $\widehat{\rho}_k^{\klein{\max}}$ and $\widehat{\rho}_k^{\klein{\min}}$ of the corresponding extreme Ritz values. The experiments predict that at the cost of computing one linear system with the tridiagonal matrix available in the form of $LDL^T$ factorization, the accuracy of the estimates can be significantly improved.

\begin{figure}
\centering
 \includegraphics[width=5.8cm]{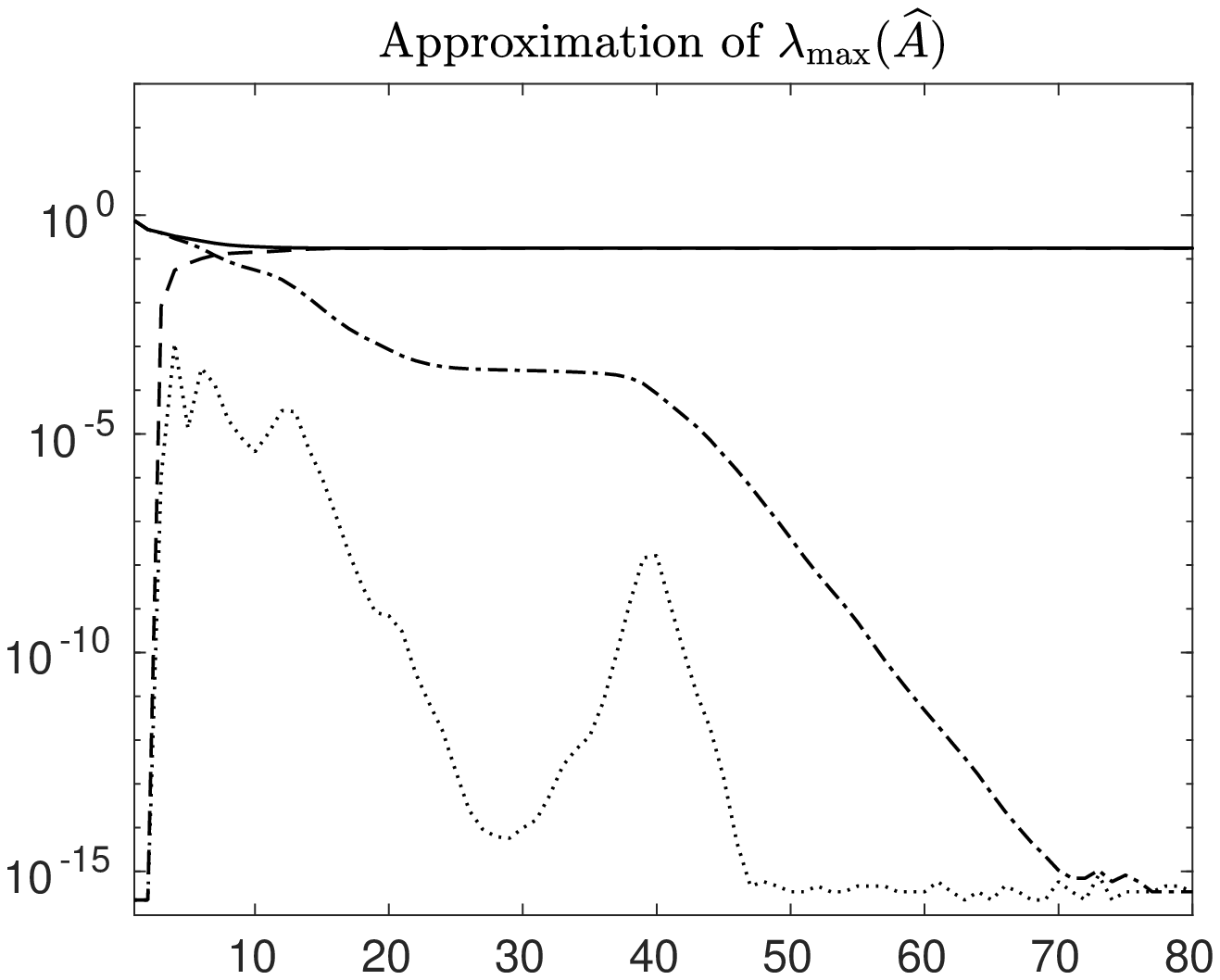}~~~
 \includegraphics[width=5.8cm]{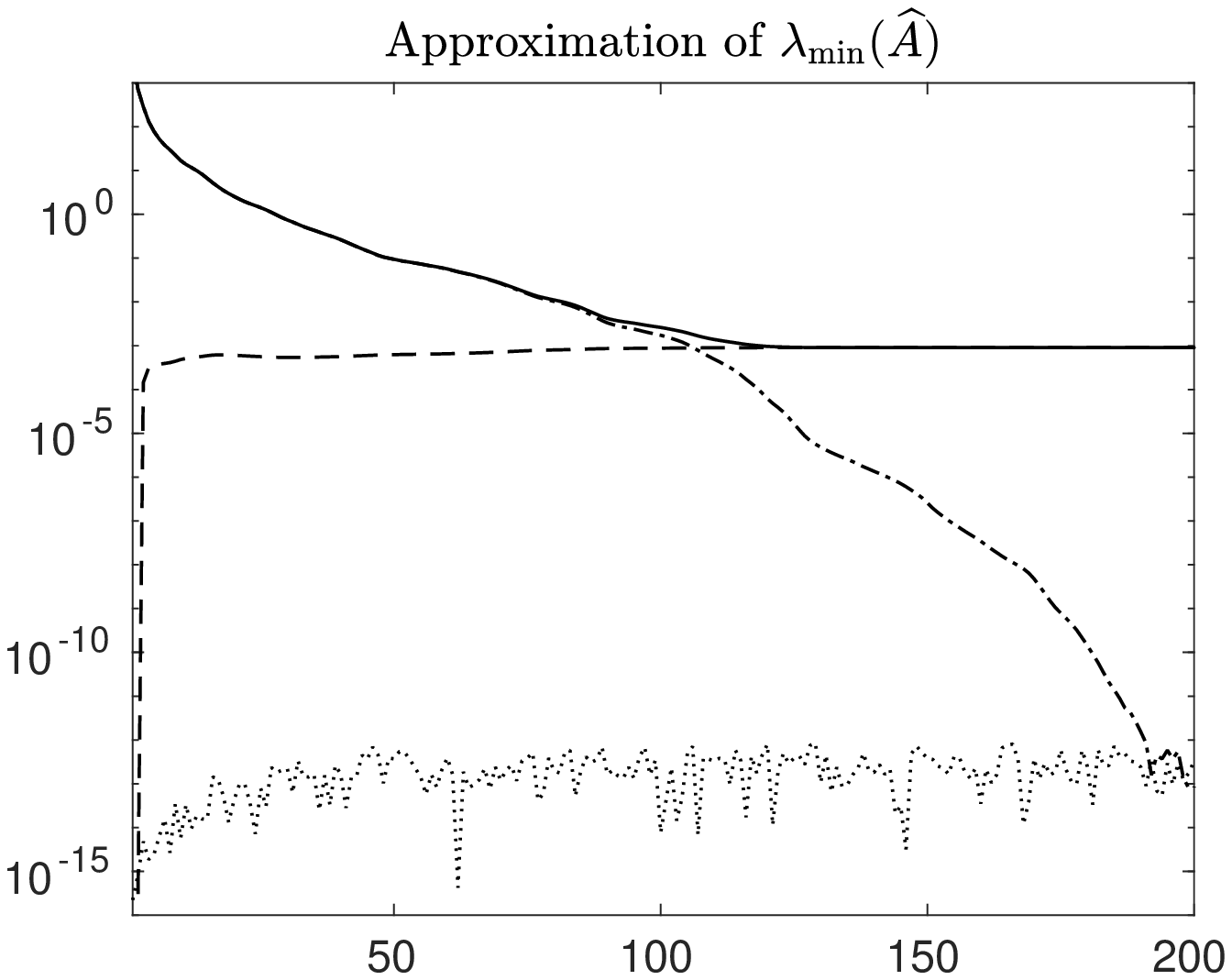}
 \caption{Approximating the extreme Ritz values for the preconditioned system {\tt Pres\_Poisson}.}
\label{fig5}
\end{figure}
\begin{figure}[!htbp]
\centering
 \includegraphics[width=5.8cm]{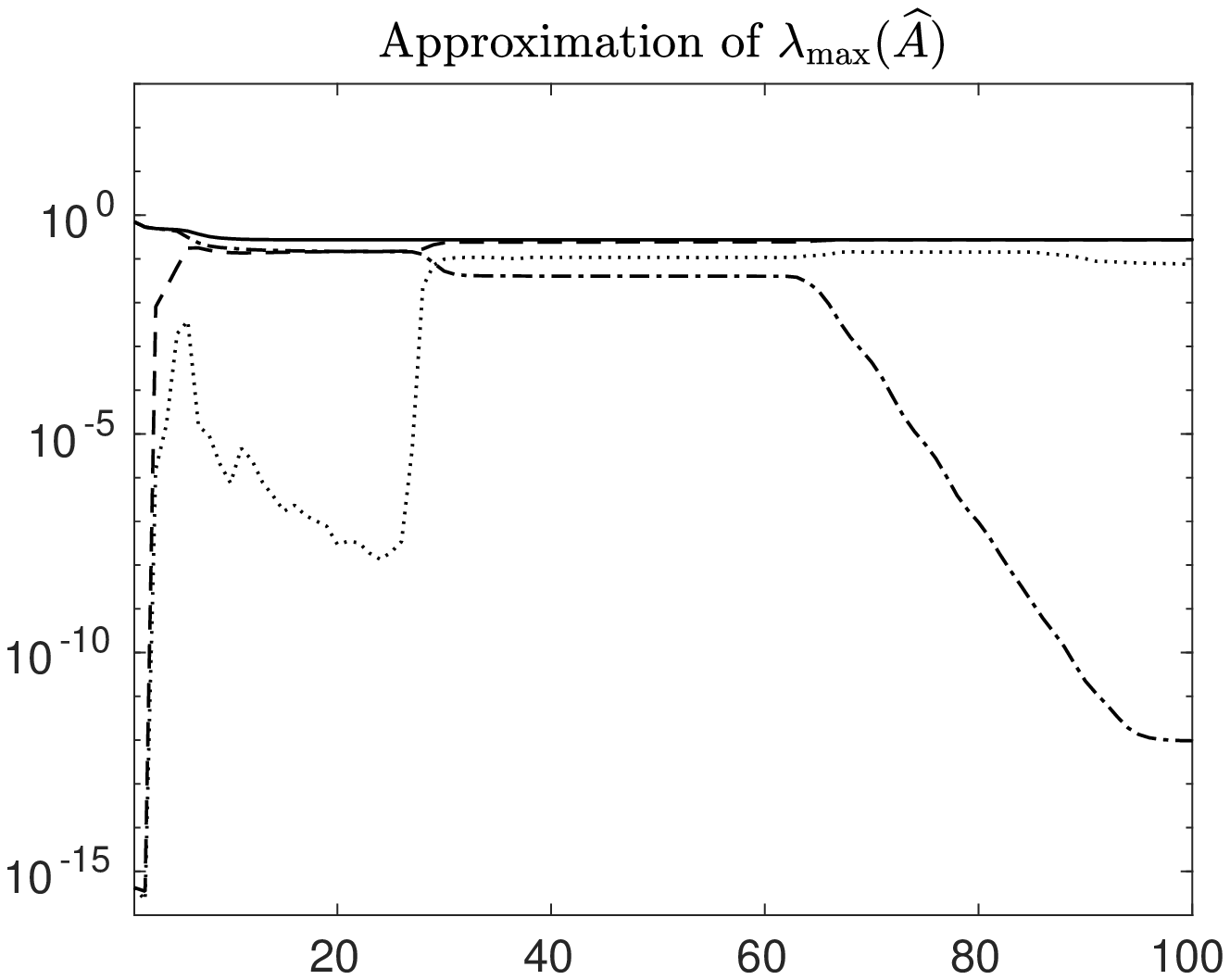}~~~
 \includegraphics[width=5.8cm]{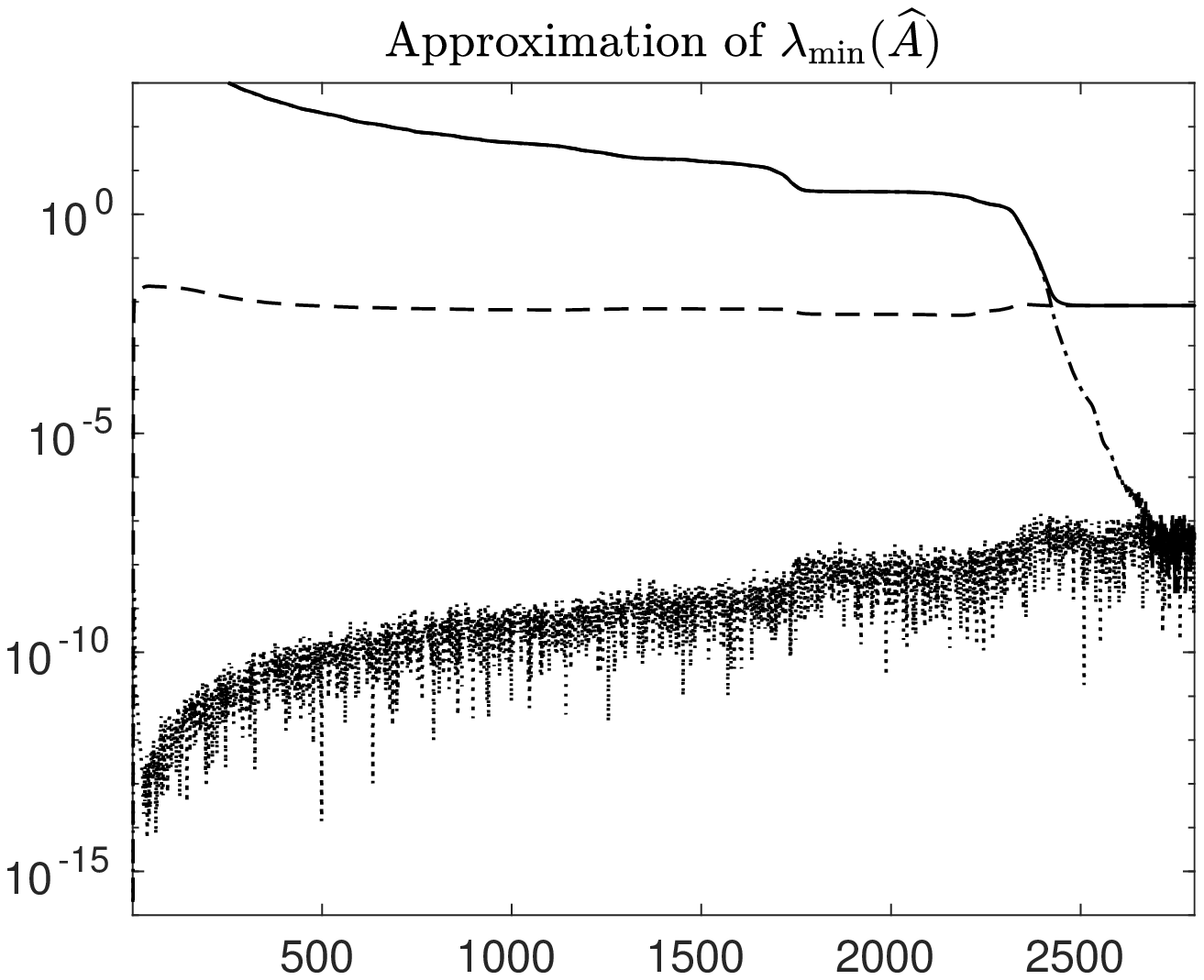}
\caption{Approximating the extreme Ritz values for the preconditioned system {\tt s3dkt3m2}.}
\label{fig6}
\end{figure}

A similar picture can be seen for preconditioned systems; see
Figures \ref{fig5}-\ref{fig6}. Recall that if we precondition the system, the extreme Ritz values approximate the extreme eigenvalues of the {\em preconditioned} matrix $\widehat{A}$.
%
%
%
As we can see, convergence of
$\theta_k^{\klein{(k)}}$ to $\lambda_{\max}(\widehat{A})$
to full precision accuracy is for the preconditioned systems significantly delayed. This is due to the fact that the preconditioned matrix has often a cluster of eigenvalues, which corresponds to the largest eigenvalue. Then, the power method as well as the Lanczos method (or CG) need more iterations to approximate the largest eigenvalue accurately. Moreover, a cluster 
of eigenvalues about the largest eigenvalue leads to a cluster of Ritz values which approximate the largest eigenvalue, and, as a consequence, 
the improved estimates $\widehat{\rho}_k^{\klein{\max}}$ (dotted curve) based on inverse iterations often do not improve the accuracy of the approximation significantly.

Similarly as in the unpreconditioned case, the cheap estimates
$\rho_k^{\klein{\max}}$ and $\rho_k^{\klein{\min}}$
approximate the corresponding Ritz values with a high relative accuracy in a few initial iterations (dashed curve), but in later iterations the relative accuracy is getting worse and stagnates on the level of about $10^{-1}$ or $10^{-2}$. As a result, one can expect that 
in later iterations, the estimates $\rho_k^{\klein{\max}}$ and $\rho_k^{\klein{\min}}$ can approximate the largest and smallest 
eigenvalues of $\widehat{A}$ also with the relative accuracy of 
$10^{-1}$ or $10^{-2}$ (solid curve).

\begin{figure}
\centering
 \includegraphics[width=5.8cm]{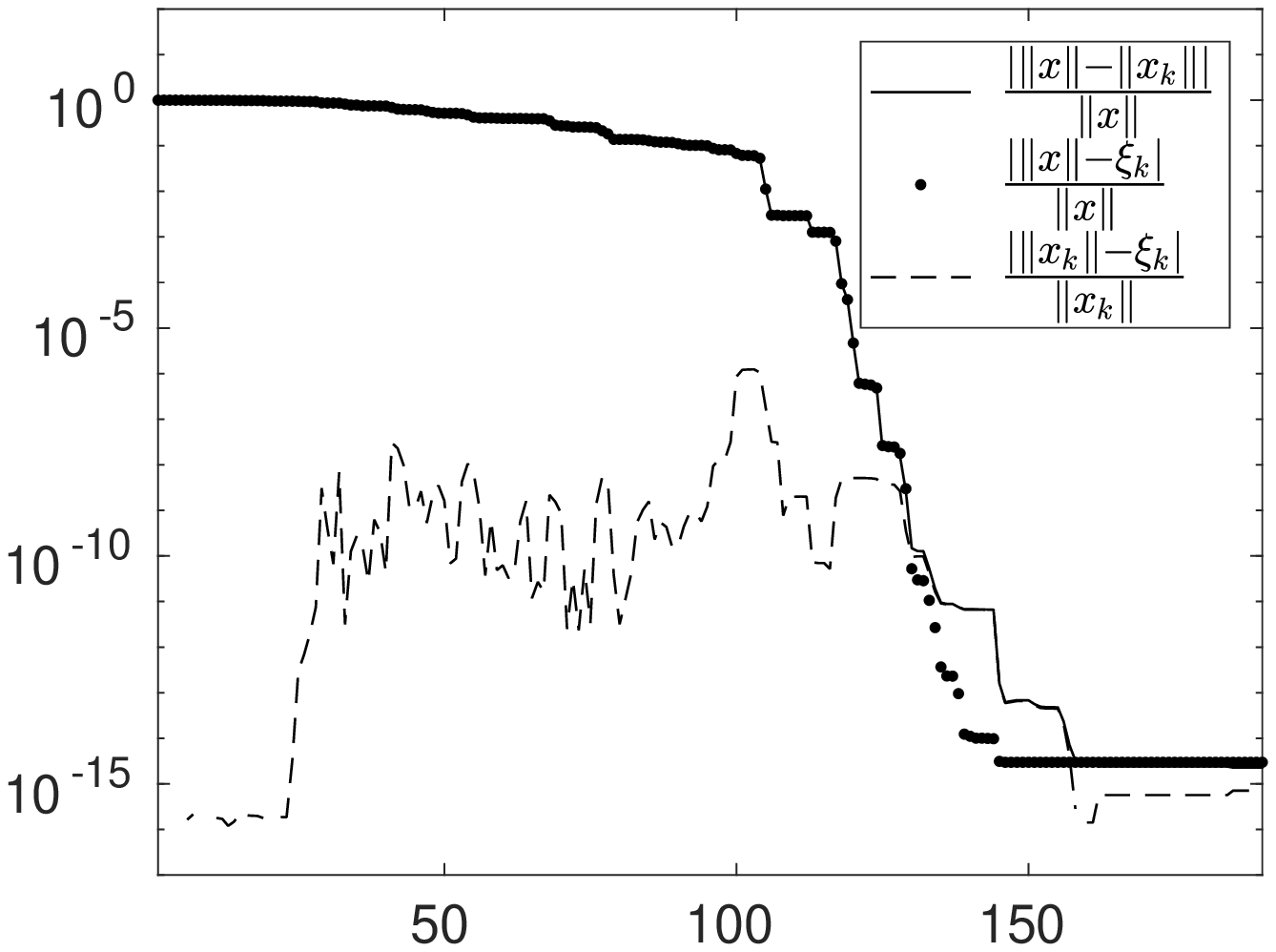}~~~
 \includegraphics[width=5.8cm]{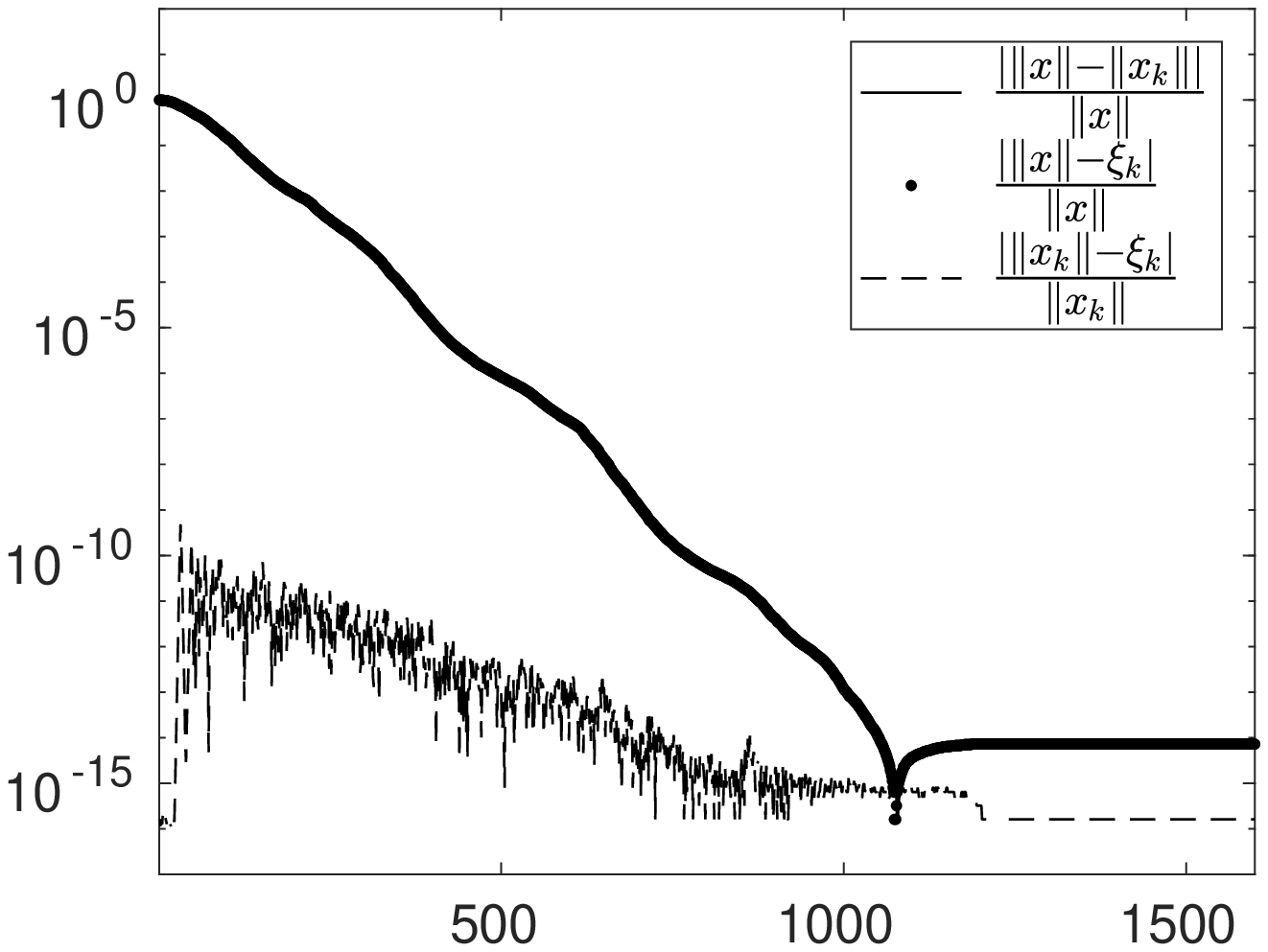}
\caption{Approximating $\|x_k\|$ using $\xi_k$ when solving
the unpreconditioned systems {\tt bcsstk01} (left part) and {\tt Pb26} (right part).}
\label{fig7}
\end{figure}
\begin{figure}[!htbp]
\centering
 \includegraphics[width=5.8cm]{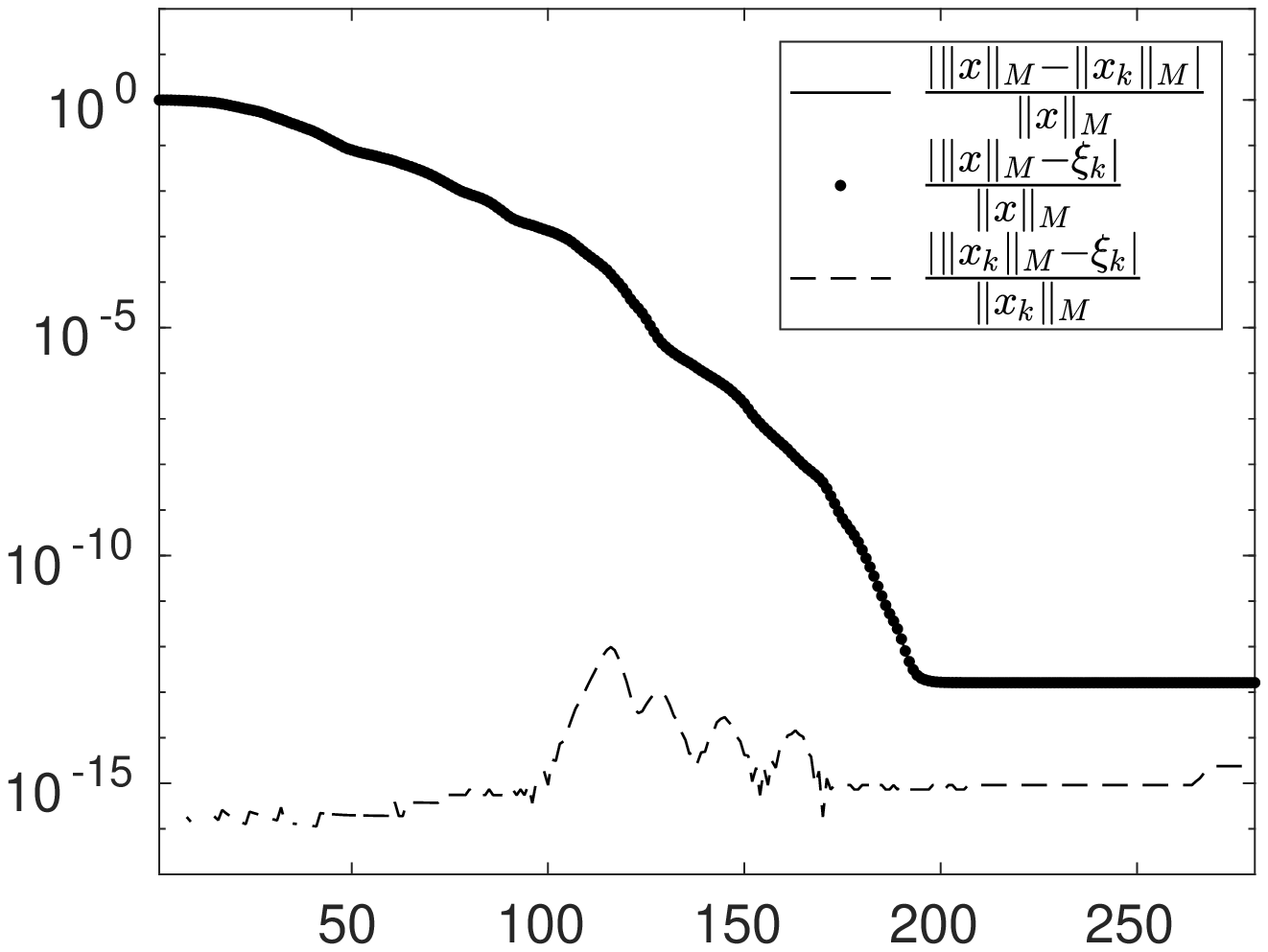}~~~
 \includegraphics[width=5.8cm]{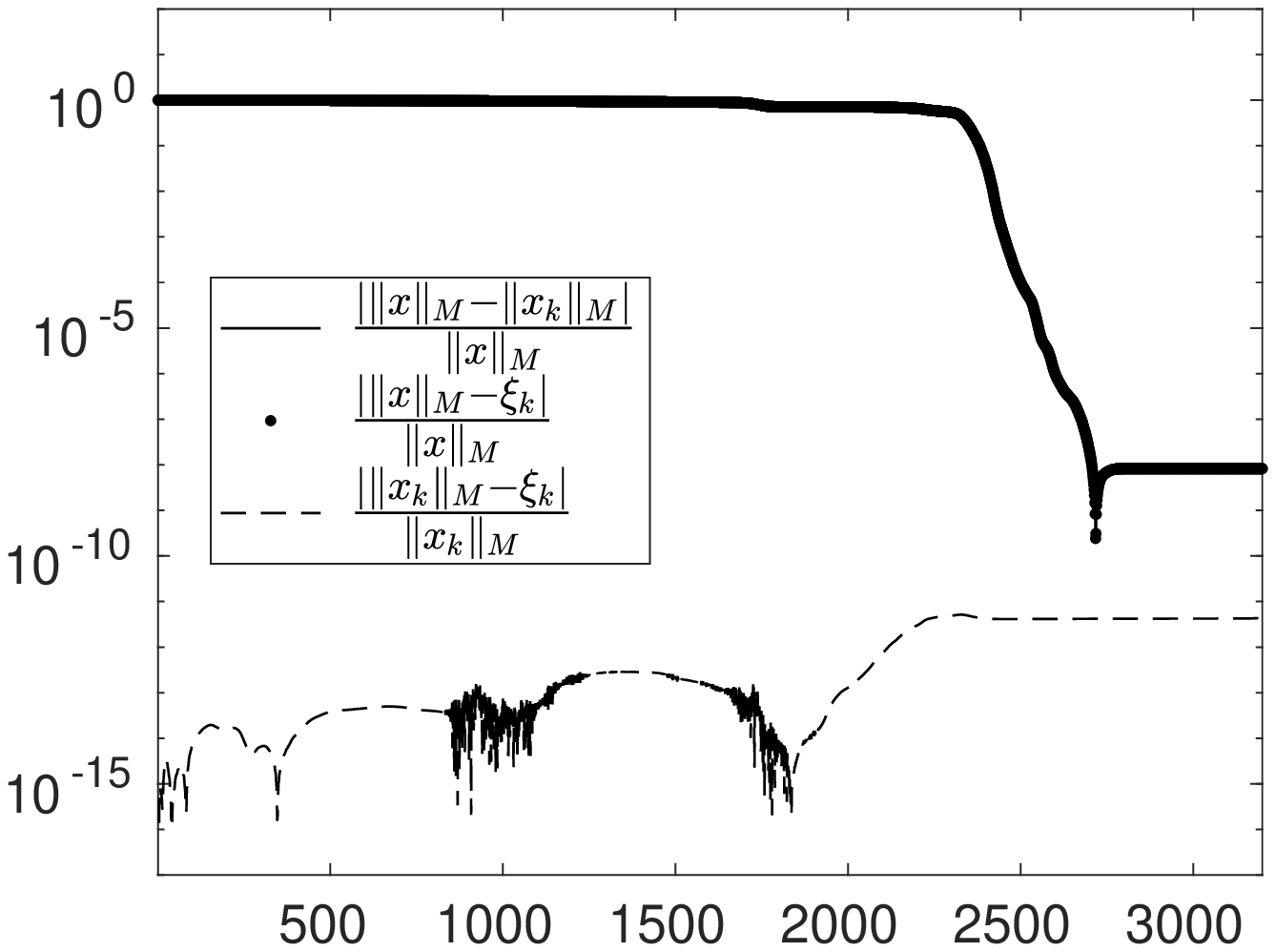} 
 \caption{Approximating $\|x_k\|_M$ using $\xi_k$ when solving
the preconditioned systems {\tt Pres\_Poisson} (left part) and {\tt s3dkt3m2} (right part).}
\label{fig8}
\end{figure}

Finally, let us test numerically, how well the quantity $\xi_k^{1/2}$  approximates $\|x_k\|$ in the unpreconditioned case, and
$\|x_k\|_M$ in the preconditioned case. Recall that
$\xi_k$ is defined by \eqref{eq:xik} and, in the experiments, we compute it
cheaply using the formulas \eqref{eq:nx1}--\eqref{eq:nx2}.

In \figurename~\ref{fig7} we consider the unpreconditioned systems {\tt bcsstk01} and {\tt Pb26}. By the dashed curve we plot the relative error of the approximation
$$
   \left| \frac{\|x_k\|-\xi_k}{\|x_k\|} \right| \quad\mbox{(dashed)}.
$$
In the left part (system {\tt bcsstk01}), the above mentioned relative error is close or below the level of $10^{-10}$, despite the severe loss of orthogonality. In other words, $\xi_k$ agrees with the approximated quantity $\|x_k\|$ to about 10 valid digits. For comparison we also plot with a solid curve the relative error of $\|x_k\|$ as an approximation of $\|x\|$, and by dots the relative error of $\xi_k$ as an approximation of $\|x\|$,
$$
   \left| \frac{\|x\|-\|x_k\|}{\|x\|} \right| \quad\mbox{(solid)},\qquad
    \left| \frac{\|x\|-\xi_k}{\|x\|} \right| \quad\mbox{(dots}).
$$
We can observe that the solid curve coincides visually with the dots until the level of $10^{-10}$ is reached. Below this level, the two curves can differ, but they are still close to each other. In the right part of the figure (system {\tt Pb26}), the relative accuracy of $\xi_k$ as an approximation of $\|x_k\|$ is even better, close to machine precision.

In \figurename~\ref{fig8} we consider systems
{\tt Pres\_Poisson} and {\tt s3dkt3m2} solved with preconditioning. Here $\xi_k$ is computed using the formulas \eqref{eq:nx1}--\eqref{eq:nx2}
from the PCG coefficients $\widehat{\gamma}_k$ and $\widehat{\delta}_k$, and it approximates $\|x_k\|_M$. Similarly as for the unpreconditioned systems, we can observe that $\xi_k$ approximates $\|x_k\|_M$ very accurately. In the considered examples, the relative errors are close to the level of machine precision.

\subsection{Approximating convergence characteristics} \label{sec:apply}
The cheap approximations to the smallest and largest Ritz values, and to the norms of approximate solutions
can be used 
to approximate various characteristics  
which provide some information about the convergence. 
In particular, in this section we concentrate on approximating 
the normwise backward error and the Gauss-Radau upper bound, for the preconditioned systems {\tt Pres\_Poisson} and {\tt s3dkt3m2}.


\begin{figure}
\centering
 \includegraphics[width=5.8cm]{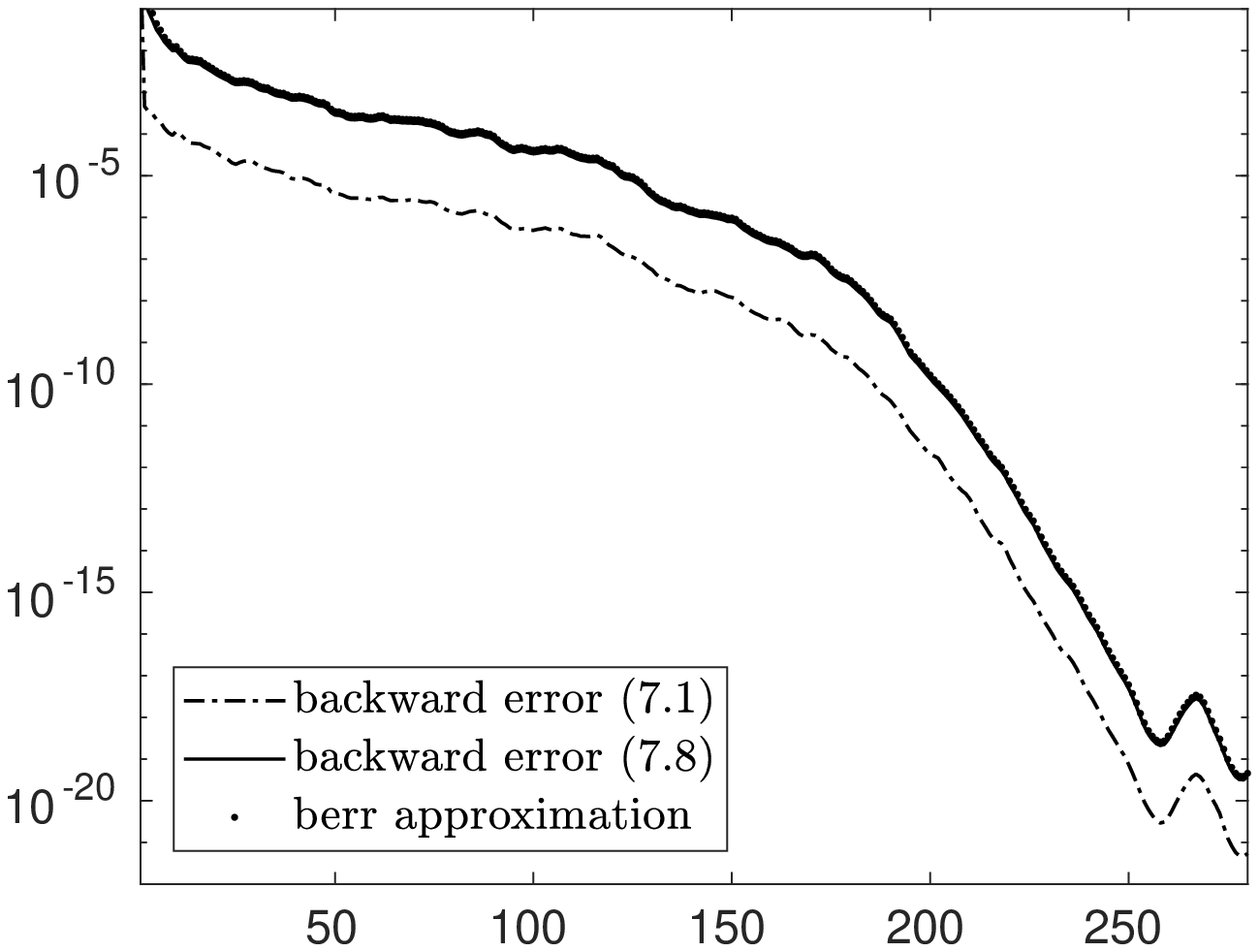}~~~
 \includegraphics[width=5.8cm]{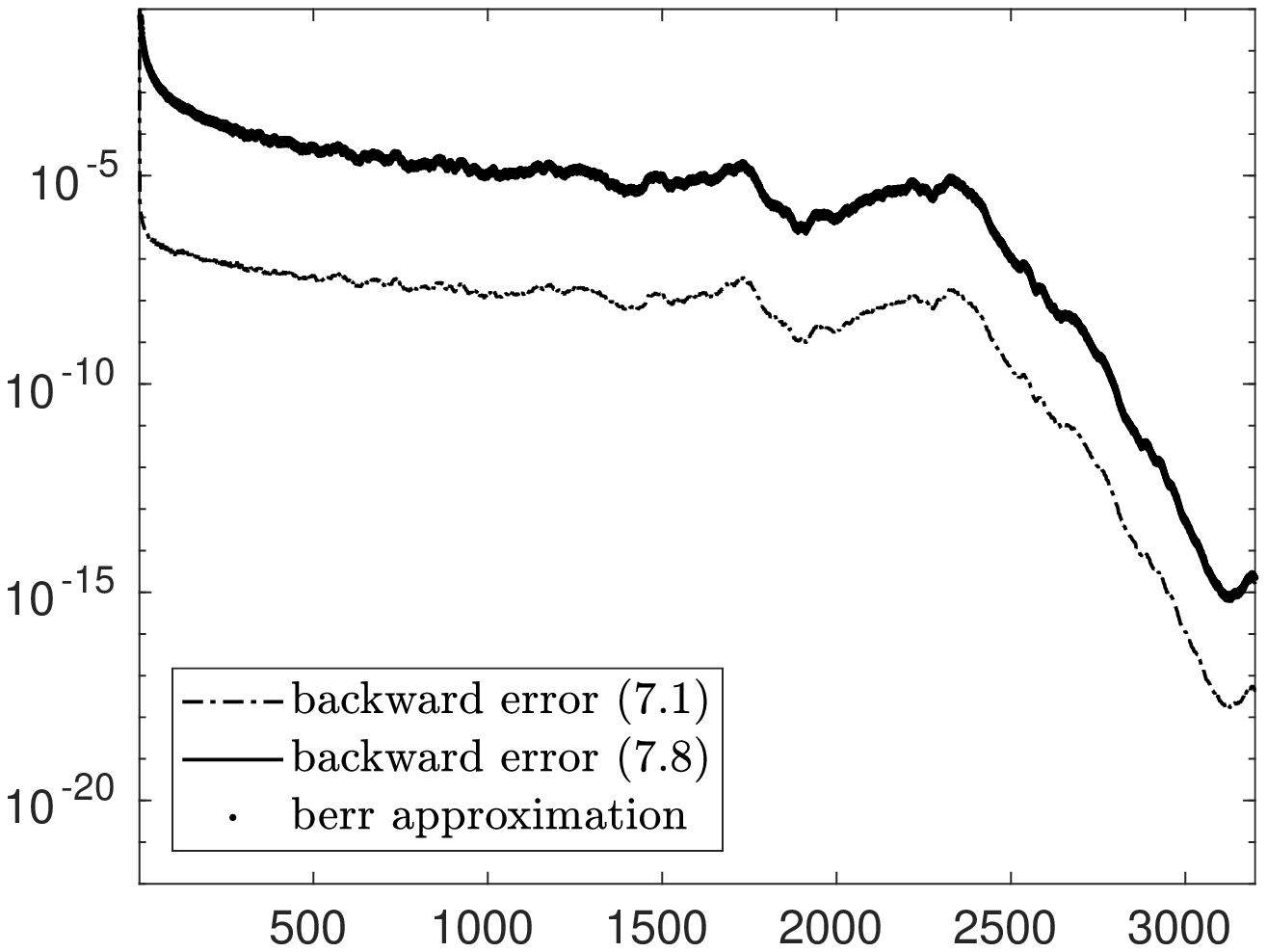}
\caption{Approximating the backward error for the preconditioned systems {\tt Pres\_Poisson}
(left part) and {\tt s3dkt3m2} (right part).}
\label{fig10}
\end{figure}

In Section~\ref{sec:be} we discussed approximation of the normwise backward error.
In \figurename~\ref{fig10} we plot the backward error \eqref{eq:bck} (solid curve) which corresponds to the original system, and the backward error 
\eqref{eq:pberr} (dash-dotted curve) which corresponds to the preconditioned system. As mentioned in Section~\ref{sec:beerPCG}, using the cheap techniques 
we can approximate the norm of the preconditioned matrix $\widehat{A}$, and the $M$-norm of the approximate solution $\|x_k\|_M$. Therefore, we can only efficiently approximate 
the backward error \eqref{eq:pberr}. The dots in \figurename~\ref{fig10} correspond to the approximations of the backward error \eqref{eq:pberr}, where $\|\widehat{A}\|$ was approximated using the incremental technique (Algorithm~\ref{alg:bidiagnorm}) and $\|x_k\|_M$ was computed 
using the formulas \eqref{eq:nx1}--\eqref{eq:nx2}. For both systems we can observe that the backward error \eqref{eq:pberr} visually coincides with its approximation.

\begin{figure}
\centering
 \includegraphics[width=5.8cm]{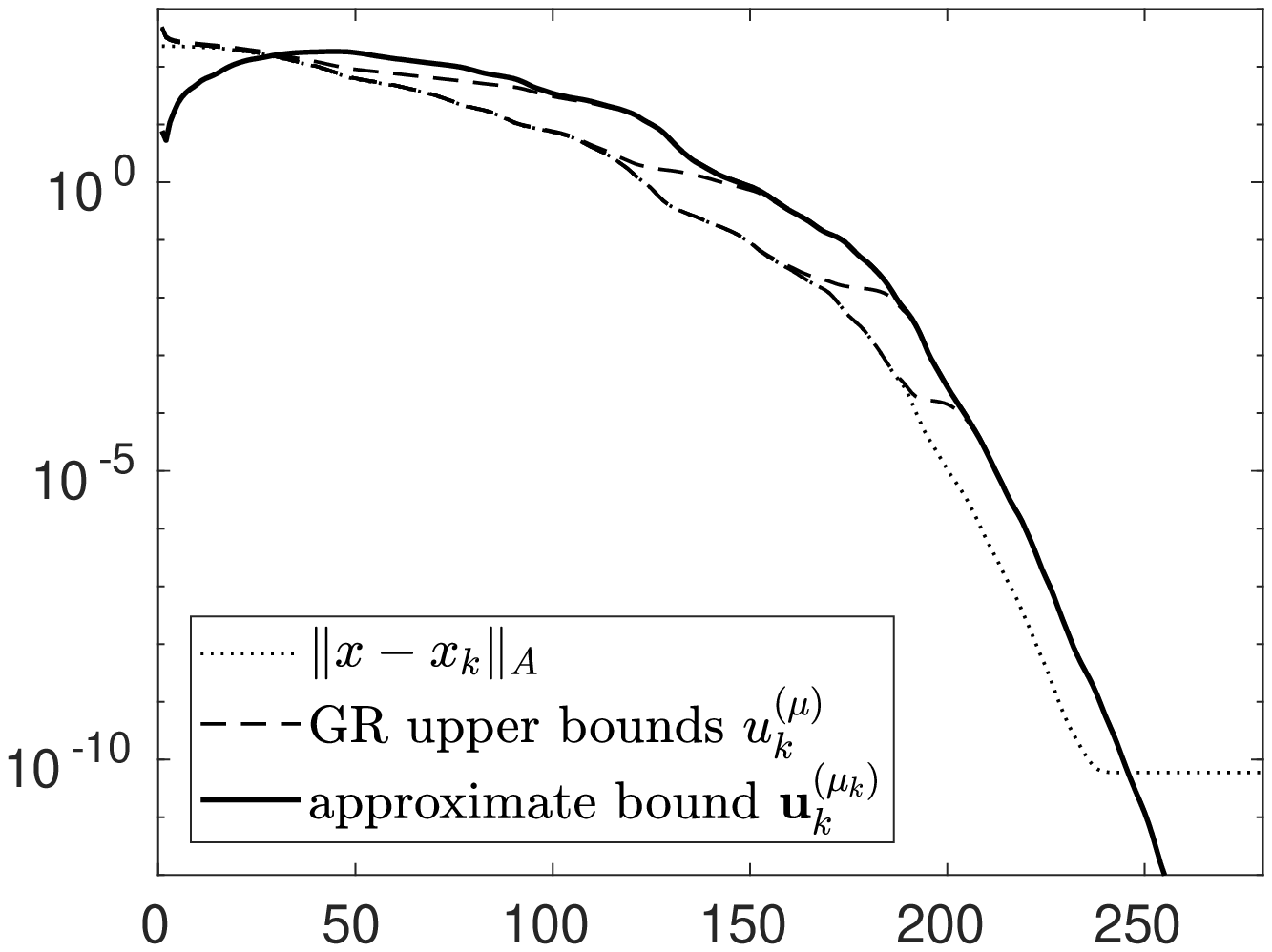}~~~
 \includegraphics[width=5.8cm]{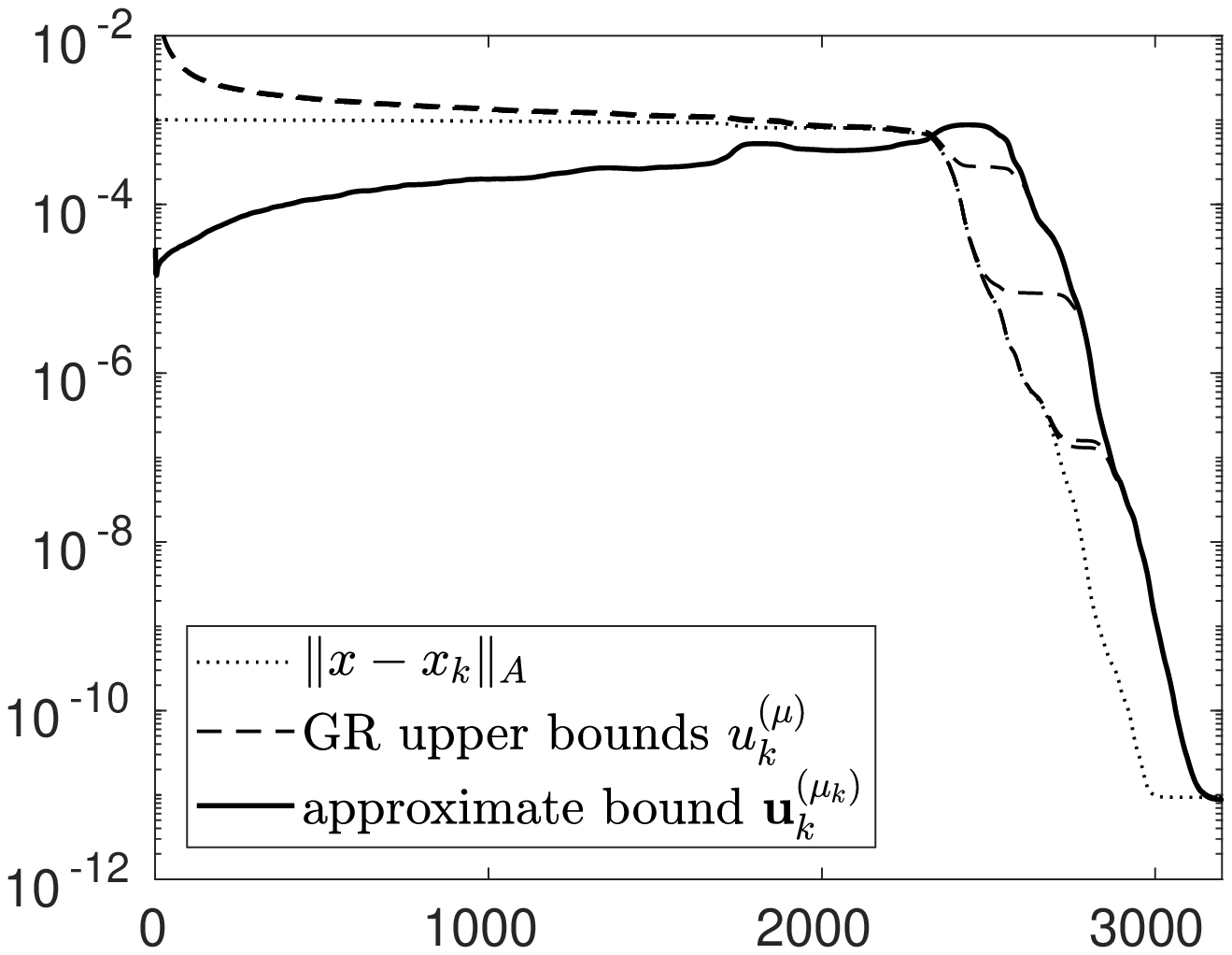}
\caption{Approximating the Gauss-Radau upper bound for the preconditioned systems {\tt Pres\_Poisson}
(left part) and {\tt s3dkt3m2} (right part).}
\label{fig9}
\end{figure}

The Gauss-Radau upper bound can be approximated using the approximate bound \eqref{eq:approx} which does not require any a priori information about the smallest eigenvalue; see Section~\ref{sec:radau}. 
In \figurename~\ref{fig9} we plot the $A$-norm of the error (dotted curve) and the Gauss-Radau upper bounds $u_k^{\smu}$ (dashed curves), where the values of $\mu $ closely approximate the smallest eigenvalue of the preconditioned matrix $\widehat{A}$ from below. Similarly as in Section~\ref{sec:bcsstk01}, we choose $\mu$ to be equal to 
$$
\frac{{\lambda}_{\min}(\widehat{A})}{(1+10^{-m})},\quad\mbox{for}\quad m=1,4,8,12.
$$
The approximate upper bound \eqref{eq:approx} 
using $\mu_k$
is plotted as a solid curve. As expected, the quantity \eqref{eq:approx} underestimates the $A$-norm of the error in the initial stage of convergence, since the smallest Ritz value is a poor approximation to the smallest eigenvalue. However, as soon as the smallest Ritz value approximates the smallest eigenvalue, the quantity \eqref{eq:approx} bounds the $A$-norm of the error from above. Moreover, in the final stage of convergence, 
the quantity \eqref{eq:approx} is as good as the Gauss-Radau upper bounds even if $\mu$ approximates ${\lambda}_{\min}(\widehat{A})$ tightly. As in the numerical example presented in Section~\ref{sec:bcsstk01}, we can observe that the Gauss-Radau upper bounds are very sensitive to the accuracy to which $\mu$ approximates  
$\lambda_{\min}(\widehat{A})$. Nevertheless, below some level, all the values of $\mu$ give visually the same upper bound, which is not very close to the $A$-norm of the error. This phenomenon appeared almost in all experiments we performed and we believe it deserves further investigation.

\begin{figure}
\centering
 \includegraphics[width=5.8cm]{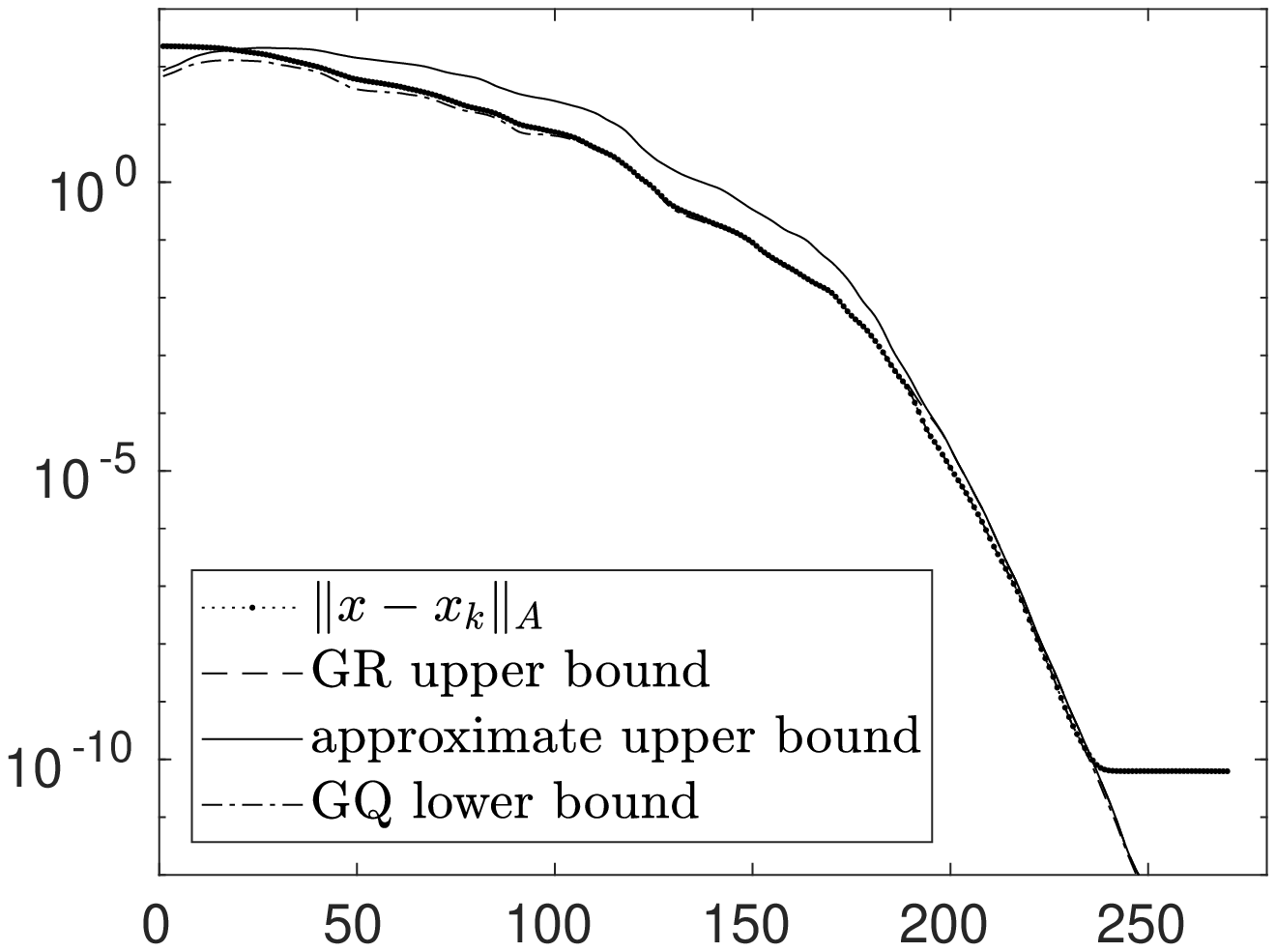}~~~
 \includegraphics[width=5.8cm]{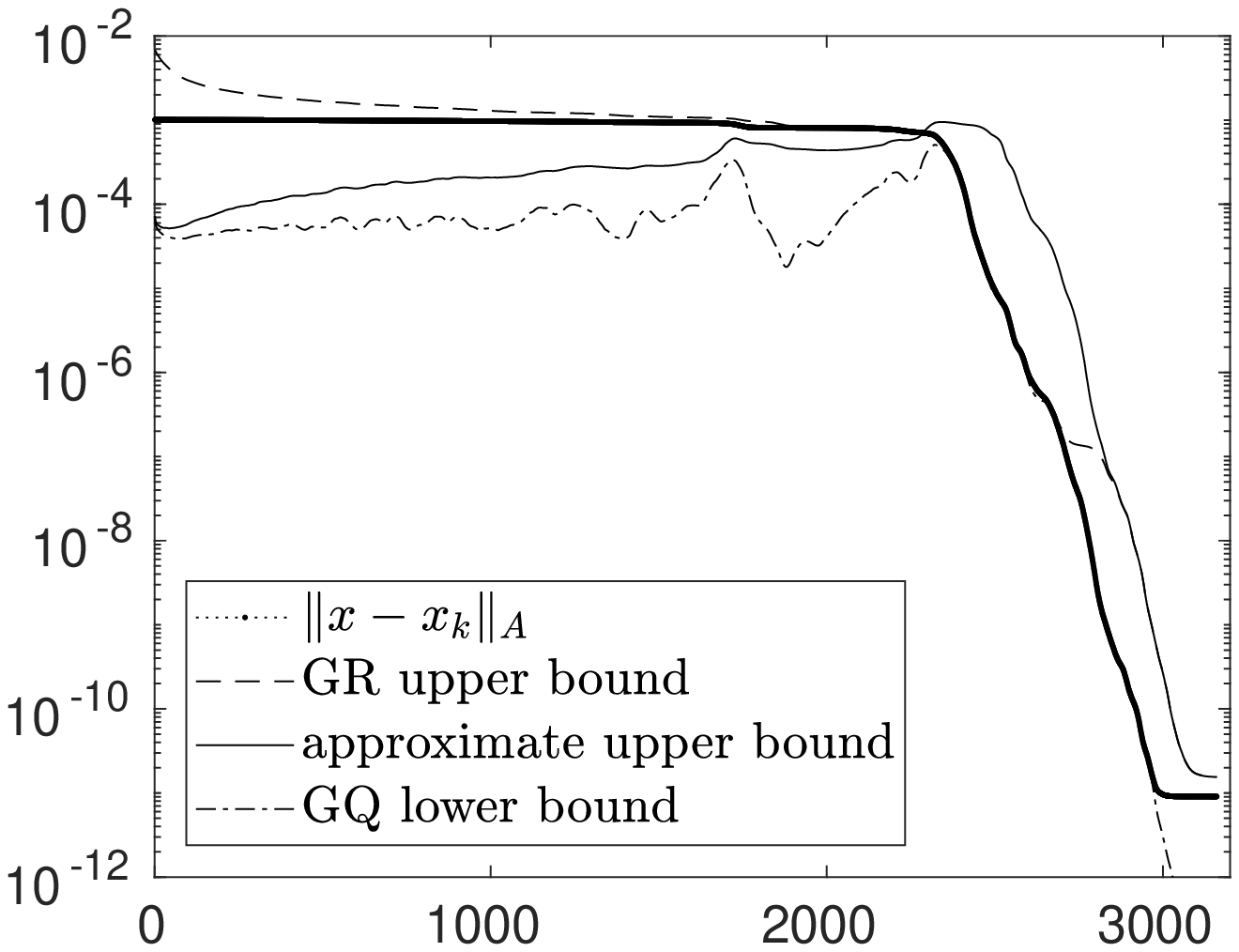}
\caption{Approximating the $A$-norm of the error using Gauss-Radau upper bound, approximate upper bound, and the Gauss lower bound for the preconditioned systems {\tt Pres\_Poisson}
(left part, $d=10$) and {\tt s3dkt3m2} (right part, $d=40$).}
\label{fig11}
\end{figure}

In the last numerical experiment (\figurename~\ref{fig11}) we choose the delay $d=10$ for the preconditioned system 
{\tt Pres\_Poisson} and $d=40$ for {\tt s3dkt3m2}.
We approximate the $A$-norm of the error (dotted curve) using the Gauss-Radau upper bound \eqref{eq:delayupper} (dashed curve) for a value of $\mu$ which closely approximates 
${\lambda}_{\min}(\widehat{A})$
from below, $\mu = {{\lambda}_{\min}(\widehat{A})}/{(1+10^{-12})}$, simulating the situation when we know ${\lambda}_{\min}(\widehat{A})$ in advance from the application. If there is no a priori information about ${\lambda}_{\min}(\widehat{A})$, one can use the approximate upper bound \eqref{eq:newbound3} (solid curve) with $\mu_{k+d}$. For comparison we also plot the Gauss lower bound based on \eqref{eq:delaylower} (dash-dotted curve).

In the left part of \figurename~\ref{fig11} we can observe that $d=10$ significantly improves all the bounds. The approximate upper bound (solid) is slightly overestimating the $A$-norm of the error $\Vert x-x_k \Vert_A$ in the initial stage of convergence. When convergence accelerates (around iteration 200), all the bounds approximate $\Vert x-x_k \Vert_A$ tightly. In \figurename~\ref{fig9} (left part) we have observed that the curves describing upper bounds are about 10 iterations delayed in the later stage of convergence. This is the reason why the choice of $d=10$ is sufficient to get good approximations to $\Vert x-x_k \Vert_A$.

In the right part of \figurename~\ref{fig11} we consider the more complicated problem with the system {\tt s3dkt3m2}. Here the choice of $d=40$  does not improve the bounds too much in the initial stagnation phase. The Gauss lower bound (dash-dotted) as well as the approximate upper bound (solid) underestimate $\Vert x-x_k \Vert_A$ significantly. The only useful bound in this phase of convergence is the Gauss-Radau upper bound (dashed) with a prescribed value of $\mu$. When the $A$-norm of the error starts to decrease (around iteration 2300), the Gauss lower bound with $d=40$ starts to be visually the same as $\Vert x-x_k \Vert_A$, until the ultimate level of accuracy is reached. This is not the case for the approximate upper bound (solid), which is significantly delayed. However, in comparison to \figurename~\ref{fig9} (right part), the approximate upper bound is moved about 40 iterations towards $\Vert x-x_k \Vert_A$. The Gauss-Radau upper bound (dashed) approximates at first
$\Vert x-x_k \Vert_A$ tightly, but, below a certain level, it starts to give the same results as the approximate upper bound, i.e., the curve is delayed. This experiment demonstrates the potential weakness of upper bounds in the final stage of convergence, and also shows a need for an adaptive choice of $d$.

\section{Conclusions}
 \label{sec:concl}
In this paper we derived a new upper bound for the $A$-norm of the error in CG. The new bound is closely related to the Gauss-Radau upper bound. While the Gauss-Radau upper bound can be very sensitive to the choice of the parameter $\mu$ which should closely approximate the smallest eigenvalue of the (preconditioned) system matrix from below, the new bound is not sensitive to the choice of $\mu$. One can use it even if $\mu$ is larger than the smallest eigenvalue, as an approximate upper bound, so that $\mu$ can be chosen as an  approximation to the smallest Ritz value.

We next developed a very cheap algorithm for approximating the smallest and largest Ritz values during the CG computations. These approximations can further be improved using inverse iterations, 
at the cost of storing the CG coefficients and solving a linear system with a tridiagonal matrix at each CG iteration. The cheap approximations to the smallest and largest Ritz values can be useful in general, e.g., to approximate almost for free the condition number of the system matrix, or to estimate the ultimate level of accuracy. In this paper, we used them to approximate the parameter $\mu$ for the new upper bound on the $A$-norm of the error, and also to approximate the 2-norm of the system matrix when computing  the normwise backward error.  

Numerical experiments predict that the approximate upper bound for the $A$-norm of the error which uses the cheap technique to approximate the smallest Ritz value is in the later stage of convergence usually as good as the Gauss-Radau upper bound for which $\mu$ has to be prescribed. We also observed that even if the smallest eigenvalue is known in advance, the Gauss-Radau upper bound looses its sharpness as the $A$-norm of the error decreases, and, below some level, it is the same as the approximate upper bound. This phenomenon is caused by the underlying finite precision Lanczos process, and it deserves additional investigation.

As further demonstrated, the quality of the lower and upper bounds can be improved using the delay parameter $d$. This technique is very promising for practical estimation of the $A$-norm of the error in CG. However, constant value of $d$ is usually not sufficient in the initial stage of convergence, and it requires too many extra steps of CG in the convergence phase. Hence, there is a need for developing a heuristic technique to choose $d$ adaptively, to reflect the required accuracy of the estimate. We believe that results of this paper can be useful in developing such a technique. 
The adaptive choice of $d$ remains a subject of our further work.


\end{document}